\documentclass[reqno]{amsart}
\usepackage[top=25truemm,bottom=25truemm,left=25truemm,right=25truemm]{geometry}
\usepackage[dvipdfmx]{graphicx}
\usepackage{bm}
\usepackage{amsmath,amssymb}
\usepackage{mathrsfs}
\usepackage{enumerate}
\usepackage{multirow}
\usepackage{cases}
\usepackage{color}

\allowdisplaybreaks

\newcommand{\ds}{\displaystyle}
\newcommand{\nn}{\nonumber}
\newcommand{\ra}{\rightarrow}

\newcommand{\Ra}{\Rightarrow}

\newcommand{\Lra}{\Leftrightarrow}

\newcommand{\leftright}[3]{\left#1#2\right#3}

\newcommand{\BrS}[1]{\leftright{(}{#1}{)}}
\newcommand{\BrM}[1]{\left\{#1\right\}}
\newcommand{\BrL}[1]{\left[#1\right]}
\newcommand{\BrA}[1]{\left|#1\right|}

\newcommand{\BrSNd}[1]{(#1)}

\newcommand{\BrANd}[1]{|#1|}

\newcommand{\Diag}{\textrm{diag}}

\newcommand{\gausssymbol}[1]{\lfloor #1 \rfloor}
\newcommand{\Dim}{d}
\newcommand{\midd}{\;\middle|\;}

\newcommand{\Set}[2]{\left\{#1\midd#2\right\}}
\newcommand{\SetNd}[2]{\{#1\mid#2\}}

\newcommand{\Meas}[1]{\BrA{#1}} 
\newcommand{\Norm}[2]{\left\|#1\right\|_{#2}} 
\newcommand{\SNorm}[2]{\left|#1\right|_{#2}} 

\newcommand{\Grad}{\nabla}
\newcommand{\Lap}{\Delta}
\newcommand{\Div}{\nabla\cdot}

\newcommand{\Dm}{\Omega}

\newcommand{\Bd}{\Gamma}

\newcommand{\TimeMax}{T}

\newcommand{\N}{N}
\newcommand{\PtSetCher}{\mathcal{X}}
\newcommand{\PtSet}{\PtSetCher_{\N}}
\newcommand{\PtSetT}[1]{\PtSetCher_{\N}^{\TimeIndex{#1}}}
\newcommand{\IndexSetCher}{\Lambda}

\newcommand{\IndexSetFlT}[1]{\IndexSetCher^{\TimeIndex{#1}}_{\rm F}}
\newcommand{\IndexSetSfT}[1]{\IndexSetCher^{\TimeIndex{#1}}_{\rm S}}
\newcommand{\IndexSetWlT}[1]{\IndexSetCher^{\TimeIndex{#1}}_{\rm W}}
\newcommand{\IndexSetAll}{\IndexSetCher_{\N}}

\newcommand{\NFl}{\N_{\rm F}}
\newcommand{\NFlT}[1]{\NFl^{\TimeIndex{#1}}}
\newcommand{\NSf}{\N_{\rm S}}
\newcommand{\NSfT}[1]{\NSf^{\TimeIndex{#1}}}

\renewcommand{\i}{i}
\renewcommand{\j}{j}
\renewcommand{\k}{k}
\renewcommand{\l}{l}
\newcommand{\PtChar}{x}
\newcommand{\Pt}[1]{\PtChar_{#1}}
\newcommand{\PtT}[2]{\PtChar_{#1}^{\TimeIndex{#2}}}

\newcommand{\h}{h}

\newcommand{\WeightCher}{w}
\newcommand{\Weight}{\WeightCher}
\newcommand{\WeightDerive}{\dot{\Weight}}

\newcommand{\WeightH}{\Weight_{\h}}
\newcommand{\WeightHDerive}{\dot{\WeightCher}_{\h}}
\newcommand{\whd}{\WeightHDerive}
\newcommand{\w}{\Weight}

\newcommand{\wh}{\WeightH}

\newcommand{\FsSolCher}{V}

\newcommand{\FsSolAppT}[1]{\FsSolCher_{\N}^{#1}}
\newcommand{\FuncCher}{v}
\newcommand{\f}{\FuncCher}

\newcommand{\PvChar}{\omega}
\newcommand{\Pv}[1]{\PvChar_{#1}}
\newcommand{\Pvi}{\Pv{\i}}

\newcommand{\PvSetCher}{\mathcal{V}}
\newcommand{\PvSet}{\PvSetCher_{\N}}
\newcommand{\PvSetT}[1]{\PvSet^{\TimeIndex{#1}}}
\newcommand{\PvT}[2]{\Pv{#1}^{\TimeIndex{#2}}}

\newcommand{\FsWeightFunc}{\mathcal{W}}
\newcommand{\gconst}{c}

\newcommand{\PvTempCher}{\widetilde{\PvChar}}
\newcommand{\PvTemp}[1]{\PvTempCher_{\i}}

\newcommand{\FsCChar}{C}

\newcommand{\FsC}[2]{\FsCChar^{#1}(#2)}

\newcommand{\TStep}{k}
\newcommand{\TStepMax}{K}
\newcommand{\TimeIndex}[1]{#1}
\newcommand{\TimeChar}{t}
\newcommand{\tk}{\TimeChar^{\TimeIndex{\TStep}}}

\newcommand{\tkp}{\TimeChar^{\TimeIndex{\TStep+1}}}

\newcommand{\Dtcher}{\tau}
\newcommand{\Dt}{\Dtcher}
\newcommand{\DtT}[1]{\Dt^{\TimeIndex{#1}}}

\newcommand{\MassSPH}[1]{m_{#1}}
\newcommand{\DensSPH}[1]{\rho_{#1}}

\newcommand{\tmax}{T}

\newcommand{\Vel}{u}

\newcommand{\VelIni}{\Vel_0}

\newcommand{\Pres}{p}

\newcommand{\Dens}{\rho}

\newcommand{\force}{f}

\newcommand{\KVisc}{\nu}

\newcommand{\mderivD}{{\rm D}}
\newcommand{\mderiv}[3]{
	\ifnum #1=1	\frac{\mderivD#3}{\mderivD{#2}}%
	\else{\frac{\mderivD^{#1}#3}{\mderivD^{#1}{#2}} %
	}\fi%
}
\newcommand{\BdSf}{\Bd_{\rm S}}
\newcommand{\BdWall}{\Bd_{\rm W}}

\newcommand{\VelAppChar}{\mathrm{\Vel}}
\newcommand{\VelApp}[2]{\VelAppChar_{#1}^{\TimeIndex{#2}}}
\newcommand{\VelPredChar}{\widetilde{\VelAppChar}}
\newcommand{\VelPred}[2]{\VelPredChar_{#1}^{\TimeIndex{#2}}}
\newcommand{\PresAppChar}{\mathrm{\Pres}}
\newcommand{\PresApp}[2]{\PresAppChar_{#1}^{\TimeIndex{#2}}}

\newcommand{\VelAppLap}[2]{\langle\Lap\VelAppChar\rangle_{#1}^{\TimeIndex{#2}}}
\newcommand{\VelPredLap}[2]{\langle\Lap\VelPredChar\rangle_{#1}^{\TimeIndex{#2}}}
\newcommand{\VelPredDiv}[2]{\langle\Div\VelPredChar\rangle_{#1}^{\TimeIndex{#2}}}
\newcommand{\PresAppGrad}[2]{\langle\Grad\PresAppChar\rangle_{#1}^{\TimeIndex{#2}}}
\newcommand{\PresAppLap}[2]{\langle\Lap\PresAppChar\rangle_{#1}^{\TimeIndex{#2}}}

\newcommand{\dN}{\mathbb{N}}

\newcommand{\dR}{\mathbb{R}}
\newcommand{\dRd}{\mathbb{R}^{\Dim}}

\newcommand{\xchar}{x}
\newcommand{\x}{\xchar}

\newcommand{\ychar}{y}
\newcommand{\y}{\ychar}

\newcommand{\dx}{d\x}
\newcommand{\dy}{d\y}

\def\deq{\mathrel{\mathop:}=}%

\newcommand{\DmFuncCher}{f}
\newcommand{\DmFunc}{\DmFuncCher}

\newcommand{\IndexSetSeq}[2]{{#1}_{#2}}
\newcommand{\IndexSetSeqLast}{\zeta}
\newcommand{\IndexSetSeqAst}[2]{\IndexSetSeq{#1}{#2}^{\ast}}
\newcommand{\IndexSetSeqLastAst}{\IndexSetSeqLast^{\ast}}

\usepackage{amsmath,amssymb}
\newtheorem{theorem}{Theorem}
\newtheorem{definition}[theorem]{Definition}
\newtheorem{lemma}[theorem]{Lemma}
\newtheorem{remark}[theorem]{Remark}
\newtheorem{cor}[theorem]{Corollary}

\newcommand{\WCofSupp}{r_0}

\usepackage{comment}
\allowdisplaybreaks

\newcommand{\ConstSemiReg}{c_0}
\newcommand{\Matrix}{A}
\newcommand{\VecSol}{\mathrm{x}}
\newcommand{\VecSouce}{b}
\newcommand{\MatrixPoisson}{\widehat{\Matrix}}
\newcommand{\MatrixDiag}{D}
\newcommand{\VecSolPoisson}{\widehat{\VecSol}}
\newcommand{\VecSoucePoisson}{\widehat{\VecSouce}}
\newcommand{\MatrixComp}[2]{a_{#1#2}}
\newcommand{\CofPositiveDifinite}{\alpha}

\newcommand{\InnerProdISPH}[3]{\left(#1,#2\right)_{#3}}
\newcommand{\NormLISPH}[3]{\Norm{#1}{0,#3}}

\newcommand{\NormHzTISPH}[3]{\SNorm{#1}{1,#3,#2}}
\newcommand{\NormHzInvTISPH}[3]{\SNorm{#1}{-1,#3,#2}}
\renewcommand{\VelIni}{a}
\newcommand{\ConstTStepCond}{\delta}
\newcommand{\DivISPH}[3]{\langle\Div{#1}\rangle_{#2}^{\TimeIndex{#3}}}
\newcommand{\GradISPH}[3]{\langle\Grad{#1}\rangle_{#2}^{\TimeIndex{#3}}}
\newcommand{\LapISPH}[3]{\langle\Lap{#1}\rangle_{#2}^{\TimeIndex{#3}}}
\newcommand{\TStepCondWCoff}{\alpha}
\newcommand{\WCoff}[1]{\beta_\Dim^{\rm #1}}

\title[Unique solvability and stability analysis for ISPH method]{Unique solvability and stability analysis for incompressible smoothed particle hydrodynamics method }

\author[Y. Imoto]{Yusuke Imoto${}^\dagger$}
\address{${}^\dagger$Tohoku Forum for Creativity, Tohoku University, 2-1-1 Katahira, Aoba-ku, Sendai 980-8577 Japan}
\email{y-imoto@tohoku.ac.jp}


\begin{document}
	
\begin{abstract}
	The incompressible smoothed particle hydrodynamics method (ISPH) is a numerical method widely used for accurately and efficiently solving flow problems with free surface effects. 
	However, to date there has been little mathematical investigation of properties such as stability or convergence for this method. 
	In this paper, unique solvability and stability are mathematically analyzed for implicit and semi-implicit schemes in the ISPH method. 
	Three key conditions for unique solvability and stability are introduced: a connectivity condition with respect to particle distribution and smoothing length, a regularity condition for particle distribution, and a time step condition. 
	The unique solvability of both the implicit and semi-implicit schemes in two- and three-dimensional spaces is established with the connectivity condition. 
	The stability of the implicit scheme in two-dimensional space is established with the connectivity and regularity conditions. 
	Moreover, with the addition of the time step condition, the stability of the semi-implicit scheme in two-dimensional space is established. 
	As an application of these results, modified schemes are developed by redefining discrete parameters to automatically satisfy parts of these conditions.
\end{abstract}

\maketitle

\section{Introduction}
The smoothed particle hydrodynamics (SPH) method \cite{gingold1977smoothed,lucy1977numerical} is a kind of numerical method for solving partial differential equations and discretizing them in space using a weighted average of interactions between particles within a neighborhood defined by a smoothing length. 
For the incompressible Navier--Stokes equations, the incompressible smoothed particle hydrodynamics (ISPH) method, by which the equations are discretized by the SPH method in space and a semi-implicit projection method \cite{chorin1968numerical,gresho1990theoryI} in time, was developed by Cummins and Rudman \cite{cummins1999sph}. 
The ISPH method has been widely used as a numerical method as it is able to accurately and efficiently solve flow problems with free surface effects \cite{asai2012stabilized,lind2012incompressible,xu2009accuracy}. 
Moreover, in order to simulate problems with high viscosity, an ISPH method that uses an implicit projection method has been developed \cite{lind2016high}. 

However, there is almost no mathematical background on properties such as stability or convergence for the ISPH method. 
Although there are a few mathematical analyses for the SPH method or related particle methods, e.g., error estimates for the SPH method with particle volumes related to the vortex method \cite{moussa2006convergence,moussa2000convergence,raviart1985analysis} and error estimates for Poisson and heat equations of a generalized particle method \cite{imoto2016phd}, their results do not directly apply to the ISPH method. 
Hence, the identification of discrete parameter conditions necessary for obtaining stable results has had to rely on experimental studies \cite{shao2003incompressible}. 

This paper establishes the mathematical properties of unique solvability and stability for implicit and semi-implicit schemes in the ISPH method. 
We introduce three key conditions for unique solvability and stability: a connectivity condition with respect to the particle distribution and smoothing length, a regularity condition for the particle distribution, and a time step condition corresponding to viscous diffusion. 
Then, we show the unique solvability of both the implicit and semi-implicit schemes in two- and three-dimensional spaces with the connectivity condition. 
We go on to prove the stability of velocity for the implicit scheme in two-dimensional space with the connectivity and regularity conditions. 
Further, we show the stability of velocity for the semi-implicit scheme in two-dimensional space with the addition of the time step condition. 
The main advantage of these results in the engineering sense is to clarify conditions required for stable computing in the ISPH method. 
As an application of these results, we introduce modified schemes with discrete parameters redefined to automatically satisfy the semi-regularity and time step conditions.

\section{Incompressible smoothed particle hydrodynamics method}
\label{sec:formulation}
Let $\Dm$ be a bounded domain in $\dRd\,(\Dim=2,3)$ with smooth boundary $\Bd$. 
The boundary $\Bd$ is divided into two parts: a wall boundary $\BdWall\subset\Bd$ and a free surface boundary $\BdSf\deq\Bd\setminus\BdWall$; see Fig.~\ref{fig:domain}.
\begin{figure}[th]
	\includegraphics[width=60mm]{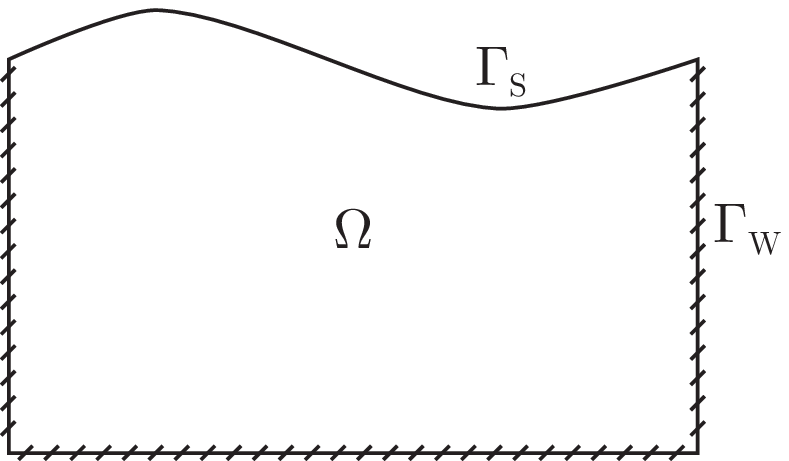}
	\caption{Domain $\Dm$ and boundaries $\BdWall$ and $\BdSf$}
	\label{fig:domain}
\end{figure} 

We consider the incompressible Navier--Stokes equations: Find $\Vel:\Dm\times(0,T)\ra\dRd$ and $\Pres:\Dm\times(0,T)\ra\dR$ such that
\begin{equation}
	\left\lbrace \begin{array}{@{\,}ll}
		\ds\mderiv{1}{t}{\Vel} =-\frac{1}{\rho}\Grad\Pres+ \nu\Lap\Vel + \force,\quad &(x,t)\in\Dm\times(0,\tmax), 
		\\
		\Div\Vel = 0, \quad&(x,t)\in\Dm\times(0,\tmax), 
		\\
		\Vel= \VelIni, \quad &x\in\Dm,~t=0, 
		\\
		\Vel = 0, \quad&(x,t)\in\BdWall\times (0,\tmax),
	\end{array}\right. 
	\label{Navier-Stokes}
\end{equation}
where $\Vel:\Dm\times (0,\tmax)\ra\dRd$, $\Pres:\Dm\times(0,\tmax)\ra\dR$, $\rho>0$, $\nu>0$, $f:\Dm\times(0,\tmax)\ra\dRd$, and $\VelIni:\Dm\ra\dRd$ denote velocity, pressure, density, kinematic viscosity, body force, and initial velocity of the fluid, respectively. 
Furthermore, $\mderivD/\mderivD t$ denotes the material derivative defined as $\mderivD/\mderivD t\deq\partial/\partial t+\Vel\cdot\nabla$. 
The unknowns are the velocity $\Vel$ and pressure $\Pres$. 
We assume the uniqueness and existence of a smooth solution for the incompressible Navier--Stokes equations \eqref{Navier-Stokes}. 

We introduce the ISPH method. 
Let $\Dt>0$ be the time step. 
Let $\TStepMax$ be $\TStepMax\deq\gausssymbol{\tmax/\Dt}$, where $\gausssymbol{\tmax/\Dt}$ denotes the greatest integer that is less than or equal to $\tmax/\Dt$. 
For $\TStep=0,1,\dots,\TStepMax$, the $\TStep$th time $\tk$ is defined as $\tk\deq\TStep\,\Dt$. 
For $\N\in\dN$, we define a particle distribution $\PtSet$ as
\begin{align}
	\PtSet \deq \Set{\Pt{\i}\in\Dm\cup\Bd}{\i=1,2,\dots,\N,~\Pt{\i}\neq\Pt{\j}\,(\i\neq\j)}.
	\label{def:PtSet}
\end{align}
We refer to $\Pt{\i}\in\PtSet$ as a particle.   
Let $\PtSetT{\k}$ and $\PtT{\i}{\k}$ be a particle distribution and an $\i$th particle at $\tk$, respectively. 
Let $\IndexSetAll\deq\{1,2,\dots,\N\}$. 
Let $\IndexSetSfT{\k}\subset\IndexSetAll$ be the index set of particles judged to be on the free surface, $\IndexSetWlT{\k}\subset\IndexSetAll$ the index set of particles on the wall boundary, and $\IndexSetFlT{\k}\subset\IndexSetAll$ the index set of the other particles. 
We refer to $\PtT{\i}{\k}\,(\i\in\IndexSetFlT{\k})$, $\PtT{\i}{\k}\,(\i\in\IndexSetSfT{\k})$, and $\PtT{\i}{\k}\,(\i\in\IndexSetWlT{\k})$ as an inner fluid particle, a surface particle, and a wall particle, respectively; see Fig.~\ref{fig:particle_distribution}. 
\begin{figure}[th]
	\includegraphics[width=100mm]{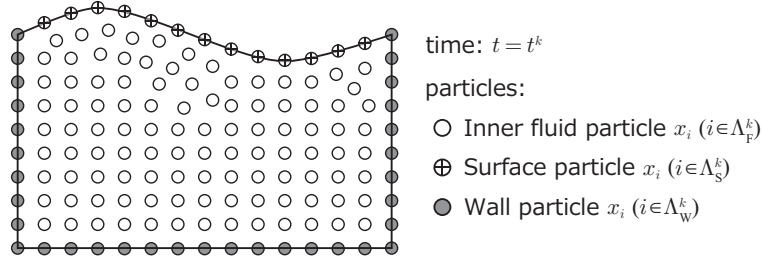}
	\caption{Particle distribution}
	\label{fig:particle_distribution}
\end{figure}

For $\N\in\dN$, we define a particle volume set $\PvSet$ as
\begin{align}
	\PvSet \deq \Set{\Pv{\i}>0}{\i=1,2,\dots,\N,~\sum_{\i=1}^\N\Pv{\i}=\Meas{\Dm}}. 
	\label{def:PvSet}
\end{align}
Here, $\Meas{\Dm}$ denotes the volume of $\Dm$. 
We refer to $\Pv{\i}\in\PvSet$ as a particle volume.  
In the SPH method, by introducing a particle density $\DensSPH{\i}$ and particle mass $\MassSPH{\i}$, the particle volume $\Pv{\i}$ is generally given as  $\Pv{\i}=\DensSPH{\i}/\MassSPH{\i}$. 

For $\w:[0,\infty)\ra\dR$, we consider the following conditions:
\begin{equation}
	\w(r)
	\begin{cases}
		>0, \quad &0<r<\WCofSupp,\\
		=0, \quad &r\geq\WCofSupp;
	\end{cases}
	\label{w_cond:compact_support}
\end{equation}
\begin{equation}
	\WeightDerive(r)
	\begin{cases}
		<0, \quad &0<r<\WCofSupp,\\
		=0, \quad &r=0 {\mbox{ or }} r\geq\WCofSupp;
	\end{cases}
	\label{w_cond:monotonically_decreasing}
\end{equation}
\begin{equation}
	\int_{\dRd}\w(|\x|)\dx=1;
	\label{w_cond:unity}
\end{equation}
\begin{equation}
	\w\in\FsC{2}{[0,\infty)}.
	\label{w_cond:2nd_smoothness}
\end{equation}
Here, $\WCofSupp$ and $\WeightDerive$ are a positive constant and the first derivative of $\w$, respectively. 
We define a function set $\FsWeightFunc$ as
\begin{equation}
	\FsWeightFunc\deq\Set{\w:[0,\infty)\ra\dR}{\mbox{$\w$ satisfies \eqref{w_cond:compact_support}--\eqref{w_cond:2nd_smoothness}}}. 
	\label{def:FsWeightFunc}
\end{equation}
We refer to $\w\in\FsWeightFunc$ as a reference weight function. 
We define a smoothing length $\h$ as a positive number that satisfies $\WCofSupp\h>\min\SetNd{|\PtT{\i}{\k}-\PtT{\j}{\k}|}{\i\neq\j}$. 
For the reference weight function $\w\in\FsWeightFunc$ and the smoothing length $\h$, a weight function $\wh:[0,\infty)\ra\dR$ is defined as
\begin{equation}
	\wh(r) \deq \frac{1}{\h^{\Dim}}\w\left(\frac{r}{\h}\right). 
	\label{def:wh}
\end{equation}
For example, in the SPH method, the following reference weight functions are often used: the cubic B-spline ($\WCofSupp=2$)
\begin{equation}
	\w(r)\deq\WCoff{cubic}
	\begin{cases}
		1-\dfrac{3}{2}r^2+\dfrac{4}{3}r^3,\quad&0\leq r<1,
		\smallskip\\
		\dfrac{1}{4}(2-r)^3,\quad&1\leq r< 2,
		\smallskip\\
		0,\quad&2\leq r,
	\end{cases}
	\label{def:weight_cubic}
\end{equation}
and the quintic B-spline ($\WCofSupp=3$)
\begin{equation}
	\w(r)\deq\WCoff{quintic}
	\begin{cases}
		(3-r)^5-6(2-r)^5+15(1-r)^5,\quad&0\leq r<1,
		\smallskip\\
		(3-r)^5-6(2-r)^5,\quad&1\leq r<2,
		\smallskip\\
		(3-r)^5,\quad&2\leq r< 3,
		\smallskip\\
		0,\quad&3\leq r.
	\end{cases}
	\label{def:weight_quintic}
\end{equation}
Here, $\WCoff{cubic}$ and $\WCoff{quintic}$ are constants dependent on $\Dim$ to satisfy the condition \eqref{w_cond:unity} and are calculated as
\begin{equation}
	\WCoff{cubic}=
	\begin{cases}
		\dfrac{10}{7\pi},\quad&\Dim=2,
		\smallskip\\
		\dfrac{1}{\pi},\quad&\Dim=3;
	\end{cases}
	\qquad
	\WCoff{quintic}=
	\begin{cases}
		\dfrac{7}{478\pi},\quad&\Dim=2,
		\smallskip\\
		\dfrac{1}{120\pi},\quad&\Dim=3. 
	\end{cases}
\end{equation}

For an index set $\Lambda\subset\IndexSetAll$ and $\k=0,1,\dots,\N$, a function space $\FsSolAppT{\k}(\Lambda)$ is defined as
\begin{equation}
	\FsSolAppT{\k}(\Lambda)\deq\BrM{\f:\{\PtT{\i}{\k}\in\PtSetT{\k}\}_{\i\in\Lambda}\ra\dR}. 
\end{equation}
For $\k=0,1,\dots,\TStepMax$, a function $\DmFunc^{\k}\in\FsSolAppT{\k}(\IndexSetAll)$ is defined as
\begin{equation}
	\DmFunc^{\k}(\PtT{\i}{\k})\deq
	\begin{cases}
		\DmFunc(\PtT{\i}{\k},\tk),\quad&\i\in\IndexSetFlT{\k}\cup\IndexSetSfT{\k},
		\\
		0,\quad&\i\in\IndexSetWlT{\k}. 
	\end{cases}
\end{equation}
Hereinafter, for a function $\f^{\k}\,(\k=0,1,\dots,\TStepMax)$ defined in $\PtSetT{\k}$, we denote $\f^{\k}(\PtT{\i}{\k})$ as $\f^{\k}_{\i}$. 
Now, we consider the following two schemes in the ISPH method. \\
{\it Implicit scheme}: find $\VelApp{}{\k}\in\FsSolAppT{\k}(\IndexSetAll)^\Dim\,(\k=0,1,\dots,\TStepMax)$ and $\PresApp{}{\k}\in\FsSolAppT{\k}(\IndexSetFlT{\k}\cup\IndexSetSfT{\k})\,(\k=1,2,\dots,\TStepMax)$ such that
\begin{equation}
	\VelApp{\i}{0} = \VelIni(\PtT{\i}{0}),\qquad \i=1,2,\dots,\N,
	\tag{I-a}\label{implicit:initial_velocity}
\end{equation}
and for $\k=0,1,\dots,\TStepMax-1$, 
\begin{align}
	&\left\lbrace \begin{array}{@{\,}r@{\,}ll}
		\dfrac{\VelPred{\i}{\k+1}-\VelApp{\i}{\k}}{\Dt} &= \KVisc\VelPredLap{\i}{\k+1}+\DmFunc_{\i}^{\k},\qquad&\i\in\IndexSetFlT{\k}\cup\IndexSetSfT{\k}, 
		\tag{I-b}\label{implicit:prediction}\\
		\VelPred{\i}{\k+1}&=0,\qquad&\i\in\IndexSetWlT{\k};
	\end{array}\right. \\
	&\left\lbrace \begin{array}{@{\,}r@{\,}ll}
		\PresAppLap{\i}{\k+1} &= \dfrac{\Dens}{\Dt}\VelPredDiv{\i}{\k+1},\hspace{9.5ex}&\i\in\IndexSetFlT{\k},
		\\
		\PresApp{\i}{\k+1} &=0,\qquad&\i\in\IndexSetSfT{\k};
	\end{array}\right. 
	\tag{I-c}\label{implicit:Poisson}\\
	&\left\lbrace \begin{array}{@{\,}r@{\,}ll}
		\dfrac{\VelApp{\i}{\k+1}-\VelPred{\i}{\k+1}}{\Dt} &=-\dfrac{1}{\Dens}\PresAppGrad{\i}{\k+1},\hspace{5.5ex}&\i\in\IndexSetFlT{\k}\cup\IndexSetSfT{\k},
		\\
		\VelApp{\i}{\k+1} &= 0,\qquad&\i\in\IndexSetWlT{\k}. 
	\end{array}\right. 
	\tag{I-d}\label{implicit:correction}
\end{align}
{\it Semi-implicit scheme}: find $\VelApp{}{\k}\in\FsSolAppT{\k}(\IndexSetAll)^\Dim\,(\k=0,1,\dots,\TStepMax)$ and $\PresApp{}{\k}\in\FsSolAppT{\k}(\IndexSetFlT{\k}\cup\IndexSetSfT{\k})\,(\k=1,2,\dots,\TStepMax)$ such that
\begin{equation}
	\VelApp{\i}{0} = \VelIni(\PtT{\i}{0}),\qquad \i=1,2,\dots,\N,
	\tag{SI-a}\label{semi-implicit:initial_velocity}
\end{equation}
and for $\k=0,1,\dots,\TStepMax-1$, 
\begin{align}
	&\left\lbrace \begin{array}{@{\,}r@{\,}ll}
		\dfrac{\VelPred{\i}{\k+1}-\VelApp{\i}{\k}}{\Dt} &= \KVisc\VelAppLap{\i}{\k}+\DmFunc_{\i}^{\k},\qquad&\i\in\IndexSetFlT{\k}\cup\IndexSetSfT{\k},
		\\
		\VelPred{\i}{\k+1}&=0,\qquad&\i\in\IndexSetWlT{\k};
	\end{array}\right. 
	\tag{SI-b}\label{semi-implicit:prediction}\\
	&\left\lbrace \begin{array}{@{\,}r@{\,}ll}
		\PresAppLap{\i}{\k+1} &= \dfrac{\Dens}{\Dt}\VelPredDiv{\i}{\k+1},\hspace{7.5ex}&\i\in\IndexSetFlT{\k},
		\\
		\PresApp{\i}{\k+1} &=0,\qquad&\i\in\IndexSetSfT{\k};
	\end{array}\right. 
	\tag{SI-c}\label{semi-implicit:Poisson}\\
	&\left\lbrace \begin{array}{@{\,}r@{\,}ll}
		\dfrac{\VelApp{\i}{\k+1}-\VelPred{\i}{\k+1}}{\Dt} &=-\dfrac{1}{\Dens}\PresAppGrad{\i}{\k+1},\hspace{3.5ex}&\i\in\IndexSetFlT{\k}\cup\IndexSetSfT{\k},
		\\
		\VelApp{\i}{\k+1} &= 0,\qquad&\i\in\IndexSetWlT{\k}. 
	\end{array}\right. 
	\tag{SI-d}\label{semi-implicit:correction}
\end{align}
Here, the approximations of the derivatives are defined as 
\begin{align}
	\VelAppLap{\i}{\k}&\deq 2\sum_{\j\neq\i}\Pv{\j} \frac{\VelApp{\i}{\k}-\VelApp{\j}{\k}}{|\PtT{\i}{\k}-\PtT{\j}{\k}|}\frac{\PtT{\i}{\k}-\PtT{\j}{\k}}{|\PtT{\i}{\k}-\PtT{\j}{\k}|}\cdot\Grad\wh(|\PtT{\i}{\k}-\PtT{\j}{\k}|), 
	\label{def:LapApp:VelApp}\\
	\VelPredLap{\i}{\k+1}&\deq 2\sum_{\j\neq\i}\Pv{\j}\frac{\VelPred{\i}{\k+1}-\VelPred{\j}{\k+1}}{|\PtT{\i}{\k}-\PtT{\j}{\k}|}\frac{\PtT{\i}{\k}-\PtT{\j}{\k}}{|\PtT{\i}{\k}-\PtT{\j}{\k}|}\cdot\Grad\wh(|\PtT{\i}{\k}-\PtT{\j}{\k}|), 
	\label{def:LapApp:VelPred}\\
	\VelPredDiv{\i}{\k+1}&\deq\sum_{\j=1}^\N\Pv{\j} \BrSNd{\VelPred{\j}{\k+1}+\VelPred{\i}{\k+1}}\cdot\Grad\wh(|\PtT{\i}{\k}-\PtT{\j}{\k}|),
	\label{def:DivAppPlus}\\
	\PresAppGrad{\i}{\k+1}&\deq\sum_{\j\in\IndexSetFlT{\k}\cup\IndexSetSfT{\k}}\Pv{\j} \BrSNd{\PresApp{\j}{\k+1}-\PresApp{\i}{\k+1}}\Grad\wh(|\PtT{\i}{\k}-\PtT{\j}{\k}|),
	\label{def:GradAppMinus}
	\\
	\PresAppLap{\i}{\k+1}&\deq 2\sum_{\j\in\IndexSetFlT{\k}\cup\IndexSetSfT{\k}\setminus\{\i\}}\Pv{\j} \frac{\PresApp{\i}{\k+1}-\PresApp{\j}{\k+1}}{|\PtT{\i}{\k}-\PtT{\j}{\k}|}\frac{\PtT{\i}{\k}-\PtT{\j}{\k}}{|\PtT{\i}{\k}-\PtT{\j}{\k}|}\cdot\Grad\wh(|\PtT{\i}{\k}-\PtT{\j}{\k}|).  
	\label{def:LapApp:PresApp}
\end{align}
In both schemes, the particles are updated by
\begin{equation}
	\PtT{\i}{\k+1}=\PtT{\i}{\k}+\Dt\VelApp{\i}{\k+1},\qquad \i=1,2,\dots,\N,~\k=0,1,\dots,\TStepMax-1. 
\end{equation}
Equations \eqref{implicit:prediction} and \eqref{semi-implicit:prediction} are referred to as the prediction step, equations \eqref{implicit:Poisson} and \eqref{semi-implicit:Poisson} as the pressure Poisson equation, and equations \eqref{implicit:correction} and \eqref{semi-implicit:correction} as the correction step. 
The difference between these schemes is only the viscous term in the prediction steps \eqref{implicit:prediction} and \eqref{semi-implicit:prediction}. 
Because \eqref{implicit:prediction} and \eqref{implicit:Poisson} are implicit, we refer to the first scheme as the implicit scheme. 
In contrast, because only the equation for pressure \eqref{semi-implicit:Poisson} is implicit, we refer to the second scheme as the semi-implicit scheme. 

\begin{remark}
	Although many approximation operators have been proposed for the ISPH method, the approximation operators \eqref{def:LapApp:VelApp}--\eqref{def:LapApp:PresApp} are chosen because they satisfy mathematical properties in Sect.~\ref{subsec:Discrete_Sobolev_norms}. 
	Analyses of the ISPH method using other approximation operators are left as future problems. 
\end{remark}

\section{Key conditions for discrete parameters}
In order to analyze the unique solvability and the stability for the implicit and semi-implicit schemes, we introduce three important conditions for discrete parameters: the connectivity, semi-regularity, and time step conditions. 
\begin{definition}[$\h$-connectivity condition]
	\label{def_connectivity}
	For the smoothing length $\h$, the particle distribution $\PtSetT{\k}$ satisfies the $\h$-connectivity condition if for all $\i\in\IndexSetFlT{\k}$, there exist sequences $\{\IndexSetSeq{\i}{\l}\}_{\l=1}^{\IndexSetSeqLast}$ and $\{\IndexSetSeqAst{\i}{\l}\}_{\l=1}^{\IndexSetSeqLastAst}\subset\IndexSetAll$ such that
	\begin{equation}
		\IndexSetSeq{\i}{1}=\i,
		\qquad 0<|\PtT{\IndexSetSeq{\i}{\l}}{\k}-\PtT{\IndexSetSeq{\i}{\l+1}}{\k}|<\WCofSupp\h\,(1\leq\l<\IndexSetSeqLast),
		\qquad
		\IndexSetSeq{\i}{\l}\in\IndexSetFlT{\k}\,(1\leq\l<\IndexSetSeqLast),\qquad
		\IndexSetSeq{\i}{\IndexSetSeqLast}\in\IndexSetSfT{\k},
		\label{cond:h-connectivity:sf}
	\end{equation}
	\begin{equation}
		\IndexSetSeqAst{\i}{1}=\i,\qquad
		0<|\PtT{\IndexSetSeqAst{\i}{\l}}{\k}-\PtT{\IndexSetSeqAst{\i}{\l+1}}{\k}|<\WCofSupp\h\,(1\leq\l<\IndexSetSeqLastAst),\qquad
		\IndexSetSeqAst{\i}{\l}\in\IndexSetFlT{\k}\,(1\leq\l<\IndexSetSeqLastAst),\qquad
		\IndexSetSeqAst{\i}{\IndexSetSeqLastAst}\in\IndexSetWlT{\k}. 
		\label{cond:h-connectivity:wl}
	\end{equation}
\end{definition}

\begin{definition}[semi-regularity condition]
	\label{def:semi-regular}
	A family $\{(\{\PtSetT{\k}\}_{\k=0}^\TStepMax, \PvSet, \h, \Dt)\}$ satisfies the semi-regularity condition if there exists a positive constant $\ConstSemiReg$ such that
	\begin{equation}
		\max_{\i=1,2,\dots,\N}\BrM{\sum_{\j=1}^\N\Pv{\j}\BrANd{\PtT{\i}{\k}-\PtT{\j}{\k}}\BrANd{\whd(|\PtT{\i}{\k}-\PtT{\j}{\k}|)}}\leq \Dim+\ConstSemiReg\Dt,\quad \k=1,2,\dots,\TStepMax. 
		\label{cond:semi-regular}
	\end{equation}
\end{definition}

\begin{definition}[time step condition]
	\label{def:time-step-cond}
	A family $\{(\{\PtSetT{\k}\}_{\k=0}^\TStepMax, \PvSet, \h, \Dt)\}$ satisfies the time step condition if there exists a constant $\ConstTStepCond\,(0<\ConstTStepCond<1)$ such that
	\begin{equation}
		\Dt\leq\dfrac{\ConstTStepCond}{2\KVisc}\BrL{\max_{\i=1,2,\dots,\N}\BrS{\sum_{\j\neq\i}\Pv{\j}\dfrac{\BrANd{\whd(|\PtT{\i}{\k}-\PtT{\j}{\k}|)}}{|\PtT{\i}{\k}-\PtT{\j}{\k}|}}}^{-1},\qquad \k=1,2,\dots,\TStepMax. 
		\label{cond:time-step}
	\end{equation}
\end{definition}

\begin{remark}
	Consider a graph $G$ whose vertex set is particle distribution $\PtSetT{\k}$ and whose edges are a pair $(\PtT{\i}{\k}, \PtT{\j}{\k})$ that satisfies $0<|\PtT{\i}{\k}-\PtT{\j}{\k}|<\WCofSupp\h$, as shown, for example, in Fig.~\ref{fig:connectivity}. 
	By Definition \ref{def_connectivity}, that the particle distribution $\PtSetT{\k}$ satisfies the $\h$-connectivity condition is equivalent to all fluid particles having a path on $G$ to a surface particle and a wall particle. 
\end{remark}

\begin{figure}[ht]
	\centering
	\includegraphics[width=120mm]{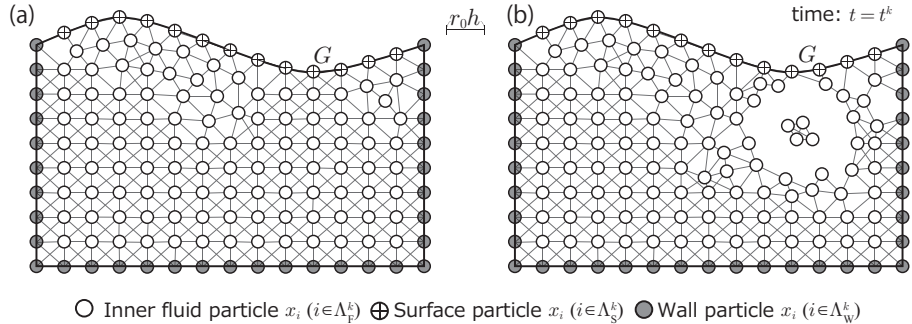}
	\caption{Particle distributions $\PtSetT{\k}$ and their graphs $G$: the left and right sides show particle distributions (a) with and (b) without satisfying the $\h$-connectivity condition, respectively}
	\label{fig:connectivity}
\end{figure}

\begin{remark}
	Because the approximation 
	\begin{equation}
		\max_{\i=1,2,\dots,\N}\BrM{\sum_{\j=1}^\N\Pv{\j}\BrANd{\PtT{\i}{\k}-\PtT{\j}{\k}}\BrANd{\whd(|\PtT{\i}{\k}-\PtT{\j}{\k}|)}}
		\approx\int_{\dRd}\BrA{\y}\BrA{\whd(|\y|)}\dy
		=\Dim 
	\end{equation}
	holds, the family $\{(\{\PtSetT{\k}\}_{\k=0}^\TStepMax, \PvSet, \h, \Dt)\}$ can satisfy the semi-regularity condition under appropriate settings of discrete parameters. 
	In particular, as the left side of \eqref{cond:semi-regular} becomes larger when a cohesion of particles occurs, the semi-regularity condition partially denotes a regularity of particle distributions. 
\end{remark}

\begin{remark}
	The time step condition \eqref{cond:time-step} corresponds to a constraint of the time step due to viscous diffusion. 
	Experimentally, the constraint is given by
	\begin{equation}
		\Dt \leq\TStepCondWCoff\dfrac{\h^2}{\KVisc},
	\end{equation}
	where the coefficient $\TStepCondWCoff$ is usually given as a value on the order of $0.1$ \cite{morris1997modeling,shao2003incompressible}. 
	In contrast, the time step condition \eqref{cond:time-step} becomes
	\begin{align}
		\Dt&\leq\dfrac{\ConstTStepCond}{2\KVisc}\BrL{\max_{\i=1,2,\dots,\N}\BrS{\sum_{\j\neq\i}\Pv{\j}\dfrac{\BrANd{\whd(|\PtT{\i}{\k}-\PtT{\j}{\k}|)}}{\BrANd{\PtT{\i}{\k}-\PtT{\j}{\k}}}}}^{-1}<\widehat{\TStepCondWCoff}\dfrac{\h^2}{\KVisc},
	\end{align}
	where $\widehat{\TStepCondWCoff}$ is defined by
	\begin{equation}
		\widehat{\TStepCondWCoff}\deq\dfrac{1}{2\h^2}\BrL{\max_{\i=1,2,\dots,\N}\BrS{\sum_{\j\neq\i}\Pv{\j}\dfrac{\BrANd{\whd(|\PtT{\i}{\k}-\PtT{\j}{\k}|)}}{\BrANd{\PtT{\i}{\k}-\PtT{\j}{\k}}}}}^{-1}. 
	\end{equation}
	Therefore, $\widehat{\TStepCondWCoff}$ corresponds to $\TStepCondWCoff$. 
	Moreover, because the approximation 
	\begin{align}
		\widehat{\TStepCondWCoff}
		&=\dfrac{1}{2\h^2}\BrL{\max_{\i=1,2,\dots,\N}\BrS{\sum_{\j\neq\i}\Pv{\j}\dfrac{\BrANd{\whd(|\PtT{\i}{\k}-\PtT{\j}{\k}|)}}{\BrANd{\PtT{\i}{\k}-\PtT{\j}{\k}}}}}^{-1}
		\nn\\
		&\approx\dfrac{1}{2\h^2}\BrS{\int_{\dRd}\dfrac{\BrA{\whd(|\y|)}}{|\y|}\dy}^{-1}\nn\\
		&=\dfrac{1}{2}\BrS{\int_{\dRd}\dfrac{\BrA{\WeightDerive(|\y|)}}{|\y|}\dy}^{-1}
	\end{align}
	holds, we can estimate the approximate value of  $\widehat{\TStepCondWCoff}$ for reference weight functions in advance. 
	When $\w$ is the cubic B-spline \eqref{def:weight_cubic}, the approximate value of $\widehat{\TStepCondWCoff}$ is calculated as 
	\begin{equation}
		\widehat{\TStepCondWCoff}\approx\dfrac{1}{2}\BrS{\int_{\dRd}\dfrac{\BrA{\WeightDerive(|\y|)}}{|\y|}\dy}^{-1}=
		\begin{cases}
			\dfrac{7}{40}=0.175,\quad&\Dim=2, 
			\smallskip\\
			\dfrac{1}{6}\approx0.167,\quad&\Dim=3, 
		\end{cases}
	\end{equation}
	and when $\w$ is the quintic B-spline \eqref{def:weight_quintic} as 
	\begin{equation}
		\widehat{\TStepCondWCoff}\approx\dfrac{1}{2}\BrS{\int_{\dRd}\dfrac{\BrA{\WeightDerive(|\y|)}}{|\y|}\dy}^{-1}=
		\begin{cases}
			\dfrac{239}{924}\approx 0.259,\quad&\Dim=2, 
			\smallskip\\
			\dfrac{1}{4}= 0.250,\quad&\Dim=3. 
		\end{cases}
	\end{equation}
	Because $\TStepCondWCoff$ is experimentally given as on the order of $0.1$, the above approximate values of $\widehat{\TStepCondWCoff}$ agree well with the experimental values. 
\end{remark}

\section{Unique solvability and stability}
\label{sec:unique_solvability_stability}
\subsection{Unique solvability}
\label{subsec:unique_solvability}
First, we show the unique solvability for the implicit and semi-implicit schemes. 
\begin{theorem}
	\label{theorem:unique_solvability}
	If particle distribution $\PtSetT{\k}$ satisfies the $\h$-connectivity condition for all $\k=0,1,\dots,\TStepMax-1$, then both the implicit and semi-implicit schemes have a unique solution. 
\end{theorem}

\begin{proof}
	Because \eqref{implicit:correction}, \eqref{semi-implicit:prediction}, and \eqref{semi-implicit:correction}  are explicit, these are clearly solvable. 
	Therefore, we prove the unique solvability of \eqref{implicit:prediction}, \eqref{implicit:Poisson}, and \eqref{semi-implicit:Poisson}. 
	
	First, we show the unique solvability of the prediction step \eqref{implicit:prediction}. 
	We fix $\k=0,1,\dots,\TStepMax$. 
	Let $\NFlT{\k}$ and $\NSfT{\k}$ be the number of inner fluid particles and  the number of surface particles, respectively, at time step $\tk$. 
	We renumber the index of particles so that $\i=1,2,\dots,\NFlT{\k}\in\IndexSetFlT{\k}$ and $\i=\NFlT{\k}+1,\NFlT{\k}+2,\dots,\NFlT{\k}+\NSfT{\k}\in\IndexSetSfT{\k}$. 
	Let $\MatrixComp{\i}{\j}\,(i,j=1,2,\dots,\N)$ be
	\begin{align}
		\MatrixComp{\i}{\j}&\deq 
		\begin{cases}
			\ds 0, \quad & i=j,\\
			\ds -2\frac{\PtT{\i}{\k}-\PtT{\j}{\k}}{|\PtT{\i}{\k}-\PtT{\j}{\k}|^2}\cdot\Grad\wh(|\PtT{\i}{\k}-\PtT{\j}{\k}|), \quad & \i\neq\j. 
		\end{cases}
	\end{align}
	We define a matrix $\Matrix\in\dR^{(\NFlT{\k}+\NSfT{\k})\times(\NFlT{\k}+\NSfT{\k})}$ and a vector $\VecSouce\in\dR^{\NFlT{\k}+\NSfT{\k}}$ respectively as
	\begin{align}
		\Matrix_{\i\j}&\deq 
		\begin{cases}
			\ds 1+\Dt\KVisc\sum_{\l=1}^{\N}\Pv{\l}\MatrixComp{\i}{\l}, \quad & \i=\j,
			\\
			\ds -\Pv{\j}\MatrixComp{\i}{\j}, \quad & \i\neq\j; 
		\end{cases}
		\\
		\VecSouce_{\i}& \deq \VelApp{\i}{\k}+\Dt\DmFunc_{\i}^{\k},\qquad\i=1,2,\dots,\NFlT{\k}+\NSfT{\k}.
	\end{align}
	Then, \eqref{implicit:prediction} is equivalent to the linear equations
	\begin{equation}
		\Matrix\VecSol=\VecSouce,
	\end{equation}
	where $\VecSol_{\i}\deq\VelPred{\i}{\k+1}\,(\i=1,2,\dots,\NFlT{\k}+\NSfT{\k})$. 
	Therefore, it is sufficient to show that $A$ is non-singular. 
	From \eqref{w_cond:monotonically_decreasing}, we have
	\begin{equation}
		-2\frac{\PtT{\i}{\k}-\PtT{\j}{\k}}{|\PtT{\i}{\k}-\PtT{\j}{\k}|^2}\cdot\Grad\wh(|\PtT{\i}{\k}-\PtT{\j}{\k}|)=-2\frac{\whd(|\PtT{\i}{\k}-\PtT{\j}{\k}|)}{|\PtT{\i}{\k}-\PtT{\j}{\k}|}\geq 0. 
	\end{equation}
	Therefore, as $\MatrixComp{\i}{\j}$ is nonnegative and $\MatrixComp{\i}{\i}=0$, we have, for $\i=1,2,\dots,\NFlT{\k}+\NSfT{\k}$, 
	\begin{equation}
		|\Matrix_{\i\i}|-\sum_{\j\neq\i}|\Matrix_{\i\j}|=1+\Dt\KVisc\sum_{\l=1}^{\N}\Pv{\l}\MatrixComp{\i}{\l}-\sum_{\j\neq\i}\Pv{\j}\MatrixComp{\i}{\j}
		=1+\Dt\KVisc\sum_{\l=\NFlT{\k}+\NSfT{\k}+1}^{\N}\Pv{\l}\MatrixComp{\i}{\l}>0. 
	\end{equation}
	Hence, $A$ is a strictly diagonally dominant matrix. 
	Thus, $A$ is non-singular. 
	
	Next, we show the unique solvability of the Poisson equations \eqref{implicit:Poisson} and \eqref{semi-implicit:Poisson}. 
	We define matrices $\MatrixPoisson, \MatrixDiag\in\dR^{\NFlT{\k}\times\NFlT{\k}}$ and a vector $\VecSoucePoisson\in\dR^{\NFlT{\k}}$ respectively as
	\begin{align}
		\MatrixPoisson_{\i\j}&\deq 
		\begin{cases}
			\ds \sum_{\l=1}^{\NSfT{\k}}\dfrac{\Pv{\l}}{\Pv{\i}}\MatrixComp{\i}{\l}, \quad & \i=\j,
			\\
			\ds -\MatrixComp{\i}{\j}, \quad & \i\neq\j;
		\end{cases}
		\nn\\
		\MatrixDiag&\deq\Diag(\Pv{\i});
		\nn\\
		\VecSoucePoisson_{\i}&\deq\dfrac{\Dens}{\Dt}\VelPredDiv{\i}{\k+1},\qquad\i=1,2,\dots,\NFlT{\k}.
	\end{align}
	Then, \eqref{implicit:Poisson} and \eqref{semi-implicit:Poisson} are equivalent to
	\begin{equation}
		\MatrixPoisson\MatrixDiag\VecSolPoisson=\VecSoucePoisson,
	\end{equation}
	where $\VecSolPoisson_{\i}\deq\PresApp{\i}{\k+1}\,(\i=1,2,\dots,\NFlT{\k})$. 
	As $\Pv{\i}>0\,(\i=1,2,\dots,\NFlT{\k})$, the diagonal matrix $\MatrixDiag$ is non-singular. 
	Therefore, it is sufficient to prove that $\MatrixPoisson$ is non-singular. 
	As $\MatrixPoisson$ is symmetric, we will prove that $\MatrixPoisson$ is a positive definite matrix. 
	For all $\CofPositiveDifinite\in\dR^{\NFlT{\k}}\setminus\{0\}$, we have
	\begin{align}
		\sum_{i,j=1}^{\NFlT{\k}} \CofPositiveDifinite_{\i}\CofPositiveDifinite_{\j} \MatrixPoisson_{\i\j}
		&= 2\sum_{1\leq i<j\leq\NFlT{\k}} \CofPositiveDifinite_{\i}\CofPositiveDifinite_{\j} \MatrixPoisson_{\i\j} + \sum_{\i=1}^{\NFlT{\k}} \CofPositiveDifinite_{\i}^2 \MatrixPoisson_{\i\i}
		\nn\\
		& = -2\sum_{1\leq i<j\leq\NFlT{\k}}\CofPositiveDifinite_{\i}\CofPositiveDifinite_{\j} \MatrixComp{\i}{\j} + \sum_{\i=1}^{\NFlT{\k}} \CofPositiveDifinite_{\i}^2 \sum_{\k=1}^{\NSfT{\k}} \frac{\Pv{\k}}{\Pv{\i}}\MatrixComp{\i}{\k}
		\nn\\
		& = \sum_{1\leq i<j\leq\NFlT{\k}}\frac{\BrS{\Pv{\j} \CofPositiveDifinite_{\i}-\Pv{\i}\CofPositiveDifinite_{\j}}^2}{\Pv{\i} \Pv{\j}}\MatrixComp{\i}{\j} + \sum_{\i=1}^{\NFlT{\k}} \CofPositiveDifinite_{\i}^2 \sum_{k=\NFlT{\k}+1}^{\NSfT{\k}} \frac{\Pv{k}}{\Pv{\i}}\MatrixComp{\i}{\k}. 
		\label{prf_poisson_unique_01}
	\end{align}
	As $\MatrixComp{\i}{\j}$ is nonnegative and $\Pv{\i}$ is positive, \eqref{prf_poisson_unique_01} is nonnegative. 
	For $a\in\dR^{\NFlT{\k}}\setminus\{0\}$, we set $\i$ such that $\CofPositiveDifinite_{\i}\neq 0$. 
	Because of the particle distribution $\PtSetT{\k}$ with $\h$-connectivity, we can take a sequence $\{\IndexSetSeq{\i}{\k}\}_{\k=1}^{\IndexSetSeqLast}$ such that the conditions given in \eqref{cond:h-connectivity:sf} hold. 
	As all terms of the last equation in \eqref{prf_poisson_unique_01} are nonnegative, we have 
	\begin{align}
		\sum_{i,j=1}^{\NFlT{\k}} \CofPositiveDifinite_{\i}\CofPositiveDifinite_{\j} \MatrixPoisson_{\i\j}
		& \geq \sum_{k=1}^{\IndexSetSeqLast-1}\frac{\BrS{\Pv{\IndexSetSeq{\i}{\k+1}} \CofPositiveDifinite_{\IndexSetSeq{\i}{\k}}-\Pv{\IndexSetSeq{\i}{\k}}\CofPositiveDifinite_{\IndexSetSeq{\i}{\k+1}}}^2}{\Pv{\IndexSetSeq{\i}{\k}}\Pv{\IndexSetSeq{\i}{\k+1}}}\MatrixComp{\IndexSetSeq{\i}{\k}}{\IndexSetSeq{\i}{\k+1}}+\frac{\Pv{\IndexSetSeq{\i}{\IndexSetSeqLast}} }{\Pv{\IndexSetSeq{\i}{\IndexSetSeqLast-1}}}\CofPositiveDifinite_{\IndexSetSeq{\i}{\IndexSetSeqLast}}^2 \MatrixComp{\IndexSetSeq{\i}{\IndexSetSeqLast-1}}{ \IndexSetSeq{\i}{\IndexSetSeqLast}}. 
		\label{proof:unieque:positive_definite_subseq}
	\end{align}
	As $|\PtT{\IndexSetSeq{\i}{\k}}{\k}-\PtT{\IndexSetSeq{\i}{\k+1}}{\k}|<\WCofSupp\h$, the value of $\MatrixComp{\IndexSetSeq{\i}{\k}}{\IndexSetSeq{\i}{\k+1}}\,(\k=1,2,\dots,\IndexSetSeqLast-1)$ is positive. 
	Thus, if the right side of \eqref{proof:unieque:positive_definite_subseq} equals zero, then $\CofPositiveDifinite_{\IndexSetSeq{\i}{\k}}=0\,(\k=1,2,\dots,\IndexSetSeqLast)$. 
	As this is inconsistent with $\CofPositiveDifinite_{\i}=\CofPositiveDifinite_{\IndexSetSeq{\i}{1}}\neq 0$, the right side of \eqref{proof:unieque:positive_definite_subseq} is positive. 
	Therefore, the matrix $\MatrixPoisson$ is a positive definite matrix. 
	Consequently, the matrix $\MatrixPoisson$ is non-singular. 
\end{proof}

\subsection{Discrete Sobolev norms and their mathematical properties}
\label{subsec:Discrete_Sobolev_norms}
Next, we introduce some notation and show certain lemmas. 
For $\Lambda\subset\IndexSetAll$, $m=1,2,\dots,\Dim$, and $\k=0,1,\dots,\N$, we define a discrete inner product $\InnerProdISPH{\cdot}{\cdot}{\Lambda}:\FsSolAppT{\k}(\Lambda)^m\times\FsSolAppT{\k}(\Lambda)^m\ra\dR$ and discrete $L^2$ norm $\NormLISPH{\cdot}{2}{\Lambda}:\FsSolAppT{\k}(\Lambda)^m\ra\dR$ as
\begin{align}
	\InnerProdISPH{\phi}{\varphi}{\Lambda}&\deq \sum_{\i\in\Lambda}\Pvi\,\phi_{\i}\cdot\varphi_{\i},\\
	\NormLISPH{\phi}{2}{\Lambda}&\deq \InnerProdISPH{\phi}{\phi}{\Lambda}^{1/2}=\BrS{\sum_{\i\in\Lambda}\Pvi\,\phi_{\i}^2}^{1/2},
\end{align}
respectively. 
Moreover, we define a discrete $H^1_0$ semi-norm $\NormHzTISPH{\cdot}{\k}{\Lambda}:\FsSolAppT{\k}(\Lambda)^m\ra\dR$ and discrete $H^{-1}_0$ semi-norm $\NormHzInvTISPH{\cdot}{\k}{\Lambda}:\FsSolAppT{\k}(\Lambda)^m\ra\dR$ as
\begin{align}
	\NormHzTISPH{\phi}{\k}{\Lambda}&\deq\BrS{\sum_{\i\in\Lambda} \Pv{\i} \sum_{\j\in\Lambda\setminus\{\i\}}\Pv{\j} \frac{|\phi_{\i}-\phi_{\j}|^2}{|\PtT{\i}{\k}-\PtT{\j}{\k}|}\BrANd{\whd(|\PtT{\i}{\k}-\PtT{\j}{\k}|)}}^{1/2},  
	\label{def:disc_H1_seminorm}\\
	\NormHzInvTISPH{\phi}{\k}{\Lambda} &\deq\sup_{\varphi\in\FsSolAppT{\k}(\Lambda)^m\setminus\{0\}}\dfrac{\InnerProdISPH{\phi}{\varphi}{\Lambda}}{\NormHzTISPH{\varphi}{\k}{\Lambda}},
	\label{def:disc_H-1_seminorm}
\end{align}
respectively. 
Then, we obtain the following lemmas:

\begin{lemma}
	\label{lem:propaties_disc_norm:01}
	For $\phi, \varphi\in\FsSolAppT{\k}(\Lambda)^\Dim\,(\Lambda\subset\IndexSetAll,\, \k=0,1,\dots,\TStepMax)$, we have 
	\begin{equation}
		\InnerProdISPH{\phi}{\varphi}{\Lambda}\leq\NormLISPH{\phi}{2}{\Lambda}\NormLISPH{\varphi}{2}{\Lambda}.
		\label{lem:ineq:innerprod:norm}
	\end{equation}
\end{lemma}

\begin{proof}
	The Cauchy--Schwarz inequality (see Appendix \ref{appendix:math_tool}) immediately yields inequality \eqref{lem:ineq:innerprod:norm}. 
\end{proof}

\begin{lemma}
	\label{lem:propaties_disc_norm:02}
	Assume the particle distribution $\PtSetT{\k}\,(\k=0,1,\dots,\TStepMax)$ satisfies the $\h$-connectivity condition and $\varphi\in\FsSolAppT{\k}(\IndexSetAll)^\Dim\,(\k=0,1,\dots,\TStepMax)$ satisfies $\varphi_\i=0$ for $\i\in\IndexSetWlT{\k}$. 
	Then, we have, for $\phi\in\FsSolAppT{\k}(\IndexSetAll)^\Dim\,(\k=0,1,\dots,\TStepMax)$, 
	\begin{equation}
		\InnerProdISPH{\phi}{\varphi}{\IndexSetFlT{\k}\cup\IndexSetSfT{\k}}=\InnerProdISPH{\phi}{\varphi}{\IndexSetAll}\leq\NormHzInvTISPH{\phi}{\k}{\IndexSetAll}\NormHzTISPH{\varphi}{\k}{\IndexSetAll}. 
		\label{lem:ineq:innerprod:seminorm}
	\end{equation}
\end{lemma}

\begin{proof}
	We first show the norm property:  $\varphi=0~\Lra~\NormHzTISPH{\varphi}{\k}{\IndexSetAll}=0$. 
	From the definition of discrete $H^1$ semi-norm \eqref{def:disc_H1_seminorm}, it is obvious that $\varphi=0~\Ra~\NormHzTISPH{\varphi}{\k}{\IndexSetAll}=0$. 
	Assume $\NormHzTISPH{\varphi}{\k}{\IndexSetAll}=0$. 
	As the particle distribution $\PtSetT{\k}$ satisfies the $\h$-connectivity condition, for any $\i\in\IndexSetFlT{\k}\cup\IndexSetSfT{\k}$, we can take a sequence $\{\IndexSetSeqAst{\i}{\k}\}_{\k=1}^{\IndexSetSeqLastAst}$ such that 
	\begin{equation}
		\IndexSetSeqAst{\i}{1}=\i,
		\qquad 
		0<|\PtT{\IndexSetSeqAst{\i}{\l}}{\k}-\PtT{\IndexSetSeqAst{\i}{\l+1}}{\k}|<\WCofSupp\h\,(1\leq\l<\IndexSetSeqLastAst),
		\qquad
		\IndexSetSeqAst{\i}{\IndexSetSeqLastAst}\in\IndexSetWlT{\k}. 
	\end{equation}
	Therefore, we have, for $\i\in\IndexSetFlT{\k}\cup\IndexSetSfT{\k}$,
	\begin{align}
		\NormHzTISPH{\varphi}{\k}{\IndexSetAll}^2
		&=\sum_{\i\in\IndexSetAll} \Pv{\i} \sum_{\j\in\IndexSetAll\setminus\{\i\}}\Pv{\j} \frac{|\varphi_{\i}-\varphi_{\j}|^2}{|\PtT{\i}{\k}-\PtT{\j}{\k}|}\BrANd{\whd(|\PtT{\i}{\k}-\PtT{\j}{\k}|)}
		\nn\\
		&\geq\sum_{\l=1}^{\IndexSetSeqLastAst-1}\Pv{\IndexSetSeqAst{\i}{\l}}\Pv{\IndexSetSeqAst{\i}{\l+1}} \frac{|\varphi_{\IndexSetSeqAst{\i}{\l}}-\varphi_{\IndexSetSeqAst{\i}{\l+1}}|^2}{|\PtT{\IndexSetSeqAst{\i}{\l}}{\k}-\PtT{\IndexSetSeqAst{\i}{\l+1}}{\k}|}\BrANd{\whd(|\PtT{\IndexSetSeqAst{\i}{\l}}{\k}-\PtT{\IndexSetSeqAst{\i}{\l+1}}{\k}|)}\geq 0. 
	\end{align}
	Then, because $\NormHzTISPH{\varphi}{\k}{\IndexSetAll}=0$ and $\Pv{\IndexSetSeqAst{\i}{\l}}\Pv{\IndexSetSeqAst{\i}{\l+1}}|\PtT{\IndexSetSeqAst{\i}{\l}}{\k}-\PtT{\IndexSetSeqAst{\i}{\l+1}}{\k}|^{-1}\BrANd{\whd(|\PtT{\IndexSetSeqAst{\i}{\l}}{\k}-\PtT{\IndexSetSeqAst{\i}{\l+1}}{\k}|)}>0$, we have $|\varphi_{\IndexSetSeqAst{\i}{\l}}-\varphi_{\IndexSetSeqAst{\i}{\l+1}}|=0$ for $\l=1,2,\dots,\IndexSetSeqLastAst-1$. 
	Moreover, because $\varphi_{\i_{\IndexSetSeqLastAst}}=0$, we obtain $\varphi_{\i}=\varphi_{\i_{1}}=0$. 
	Because $\i\in\IndexSetFlT{\k}\cup\IndexSetSfT{\k}$ is arbitrary, we obtain $\varphi=0$. 
	
	Next, we show \eqref{lem:ineq:innerprod:seminorm}. 
	As $\varphi_\i=0$ for $\i\in\IndexSetWlT{\k}$, we have $\InnerProdISPH{\phi}{\varphi}{\IndexSetFlT{\k}\cup\IndexSetSfT{\k}}=\InnerProdISPH{\phi}{\varphi}{\IndexSetAll}$. 
	When $\NormHzTISPH{\varphi}{\k}{\IndexSetAll}=0$, the norm property, $\varphi=0~\Lra~\NormHzTISPH{\varphi}{\k}{\IndexSetAll}=0$, yields
	\begin{equation}
		\InnerProdISPH{\phi}{\varphi}{\IndexSetAll}=\NormHzInvTISPH{\phi}{\k}{\IndexSetAll}\NormHzTISPH{\varphi}{\k}{\IndexSetAll}=0. 
	\end{equation}
	When $\NormHzTISPH{\varphi}{\k}{\IndexSetAll}\neq0$, from the definition of discrete $H^{-1}$ semi-norm \eqref{def:disc_H-1_seminorm},  we obtain 
	\begin{align}
		\NormHzInvTISPH{\phi}{\k}{\Lambda}
		=\sup_{\widehat\varphi\in\FsSolAppT{\k}(\Lambda)^m\setminus\{0\}}\dfrac{\InnerProdISPH{\phi}{\widehat\varphi}{\Lambda}}{\NormHzTISPH{\widehat\varphi}{\k}{\Lambda}}\geq\dfrac{\InnerProdISPH{\phi}{\varphi}{\Lambda}}{\NormHzTISPH{\varphi}{\k}{\Lambda}}. 
	\end{align}
	Therefore, we conclude \eqref{lem:ineq:innerprod:seminorm}. 
\end{proof}

\begin{lemma}
	\label{lem:propaties_disc_norm:03}
	For $\phi\in\FsSolAppT{\k}(\Lambda)^\Dim$ and $\psi\in\FsSolAppT{\k}(\Lambda)\,(\Lambda\subset\IndexSetAll,\, \k=0,1,\dots,\TStepMax)$, we have 
	\begin{equation}
		\InnerProdISPH{\psi}{\DivISPH{\phi}{}{\k}}{\Lambda}
		=-\InnerProdISPH{\GradISPH{\psi}{}{\k}}{\phi}{\Lambda},
		\label{lem:propaties_disc_norm:integral_of_parts}
	\end{equation}
	\begin{equation}
		-\InnerProdISPH{\psi}{\LapISPH{\psi}{}{\k}}{\Lambda}=\NormHzTISPH{\psi}{\k}{\Lambda}^2.
		\label{lem:propaties_disc_norm:equiv_norm}
	\end{equation}
	Here, these approximations of derivatives are defined as
	\begin{align}
		\DivISPH{\phi}{}{\k}&\deq\sum_{\j\in\Lambda}\Pv{\j} \BrS{\phi_{\j}+\phi_{\i}}\cdot\Grad\wh(|\PtT{\i}{\k}-\PtT{\j}{\k}|),
		\label{def:DivISPH}\\
		\GradISPH{\psi}{}{\k}&\deq\sum_{\j\in\Lambda}\Pv{\j} \BrS{\psi_{\j}-\psi_{\i}}\Grad\wh(|\PtT{\i}{\k}-\PtT{\j}{\k}|),
		\label{def:GradISPH}\\
		\LapISPH{\psi}{}{\k}&\deq 2\sum_{\j\in\Lambda\setminus\{\i\}}\Pv{\j} \frac{\psi_{\i}-\psi_{\j}}{|\PtT{\i}{\k}-\PtT{\j}{\k}|}\frac{\PtT{\i}{\k}-\PtT{\j}{\k}}{|\PtT{\i}{\k}-\PtT{\j}{\k}|}\cdot\Grad\wh(|\PtT{\i}{\k}-\PtT{\j}{\k}|).  
		\label{def:LapISPH}
	\end{align}
\end{lemma}

\begin{proof}
	First, we prove \eqref{lem:propaties_disc_norm:integral_of_parts}. 
	As $\Grad\wh(|\PtT{\i}{\k}-\PtT{\j}{\k}|)=-\Grad\wh(|\PtT{\j}{\k}-\PtT{\i}{\k}|)$, we have
	\begin{align}
		\InnerProdISPH{\psi}{\DivISPH{\phi}{}{\k}}{\Lambda}
		&=\sum_{\i\in\Lambda}\Pv{\i}\psi_{\i}\sum_{\j\in\Lambda}\Pv{\j}\BrS{\phi_{\j}+\phi_{\i}}\cdot\Grad\wh(|\PtT{\i}{\k}-\PtT{\j}{\k}|)\nn\\
		&=\sum_{\i\in\Lambda}\sum_{\j\in\Lambda}\Pv{\i}\Pv{\j}\psi_{\i}\BrS{\phi_{\j}+\phi_{\i}}\cdot\Grad\wh(|\PtT{\i}{\k}-\PtT{\j}{\k}|)
		\nn\\
		&=\dfrac{1}{2}\sum_{\i\in\Lambda}\sum_{\j\in\Lambda}\Pv{\i}\Pv{\j}(\psi_{\i}-\psi_{\j})\BrS{\phi_{\j}+\phi_{\i}}\cdot\Grad\wh(|\PtT{\i}{\k}-\PtT{\j}{\k}|)
		\nn\\
		&=\sum_{\i\in\Lambda}\Pv{\i}\phi_{\i}\cdot\sum_{\j\in\Lambda}\Pv{\j}(\psi_{\i}-\psi_{\j})\Grad\wh(|\PtT{\i}{\k}-\PtT{\j}{\k}|)
		\nn\\
		&=-\InnerProdISPH{\GradISPH{\psi}{}{\k}}{\phi}{\Lambda}. 
	\end{align}
	Next, we prove \eqref{lem:propaties_disc_norm:equiv_norm}. 
	Let 
	\begin{equation}
		J_{\i\j}\deq
		\begin{cases}
			0,\quad&\i=\j,
			\\
			\ds\frac{\PtT{\i}{\k}-\PtT{\j}{\k}}{|\PtT{\i}{\k}-\PtT{\j}{\k}|^2}\cdot\Grad\wh(|\PtT{\i}{\k}-\PtT{\j}{\k}|),\quad&\i\neq\j. 
		\end{cases}
	\end{equation}
	Because $J_{\i\j}=-J_{\j\i}$, we obtain
	\begin{align}
		-\InnerProdISPH{\psi}{\LapISPH{\psi}{}{\k}}{\Lambda}
		&=2\sum_{\i\in\Lambda}\Pv{\i}\psi_{\i}\sum_{\j\in\Lambda}\Pv{\j}\BrS{\psi_{\i}-\psi_{\j}}J_{\i\j}
		\nn\\
		&=2\sum_{\i\in\Lambda}\sum_{\j\in\Lambda}\Pv{\i}\Pv{\j}\psi_{\i}\BrS{\psi_{\i}-\psi_{\j}}J_{\i\j}
		\nn\\
		&=\sum_{\i\in\Lambda}\sum_{\j\in\Lambda}\Pv{\i}\Pv{\j}\BrS{\psi_{\i}-\psi_{\j}}^2J_{\i\j}
		\nn\\
		&=\NormHzTISPH{\psi}{\k}{\Lambda}^2. 
	\end{align}
\end{proof}

\begin{lemma}
	\label{lem:inequarity_disc_norm}
	Assume that a family $\{(\{\PtSetT{\k}\}_{\k=1}^{\TStepMax}, \PvSet, \h, \Dt)\}$ satisfies the semi-regularity condition with $\ConstSemiReg$. 
	Then, for $\psi\in\FsSolAppT{\k}(\Lambda)\,(\Lambda\subset\IndexSetAll,\, \k=0,1,\dots,\TStepMax)$, we have
	\begin{equation}
		\NormLISPH{\GradISPH{\psi}{}{\k}}{2}{\Lambda}^2
		\leq(\Dim+\ConstSemiReg\Dt)\NormHzTISPH{\psi}{\k}{\Lambda}^2. 
		\label{eq:lem:inequarity_disc_norm}
	\end{equation}
	Here, $\GradISPH{\psi}{}{\k}$ is defined as in \eqref{def:GradISPH}. 
\end{lemma}

\begin{proof}
	From 
	\begin{equation}
		\Grad\wh(|\PtT{\i}{\k}-\PtT{\j}{\k}|)=\dfrac{\PtT{\i}{\k}-\PtT{\j}{\k}}{|\PtT{\i}{\k}-\PtT{\j}{\k}|}\whd(|\PtT{\i}{\k}-\PtT{\j}{\k}|)
	\end{equation}
	and the Cauchy--Schwarz inequality (see Appendix \ref{appendix:math_tool}), we have 
	\begin{align}
		\NormLISPH{\GradISPH{\psi}{}{\k}}{2}{\Lambda}^2&= \sum_{\i\in\Lambda}\Pvi\BrS{\sum_{\j\in\Lambda}\Pv{\j} \BrS{\psi_{\j}-\psi_{\i}}\Grad\wh(|\PtT{\i}{\k}-\PtT{\j}{\k}|)}^2
		\nn\\
		&=\sum_{\i\in\Lambda}\Pvi\BrS{-\sum_{\j\in\Lambda}\Pv{\j} \BrS{\psi_{\j}-\psi_{\i}}\dfrac{\PtT{\i}{\k}-\PtT{\j}{\k}}{|\PtT{\i}{\k}-\PtT{\j}{\k}|}\BrANd{\whd(|\PtT{\i}{\k}-\PtT{\j}{\k}|)}}^2
		\nn\\
		&\leq\sum_{\i\in\Lambda}\Pvi\BrS{\sum_{\j\in\Lambda}\Pv{\j}|\psi_{\j}-\psi_{\i}|\BrANd{\whd(|\PtT{\i}{\k}-\PtT{\j}{\k}|)}}^2
		\nn\\
		&\leq\sum_{\i\in\Lambda}\Pvi\BrS{\sum_{\j\in\Lambda}\Pv{\j}\frac{|\psi_{\j}-\psi_{\i}|^2}{|\PtT{\i}{\k}-\PtT{\j}{\k}|}\BrANd{\whd(|\PtT{\i}{\k}-\PtT{\j}{\k}|)}}\BrS{\sum_{\j=1}^\N\Pv{\j}|\PtT{\i}{\k}-\PtT{\j}{\k}|\BrANd{\whd(|\PtT{\i}{\k}-\PtT{\j}{\k}|)}}. 
	\end{align}
	As the family $\{(\{\PtSetT{\k}\}_{\k=1}^{\TStepMax}, \PvSet, \h, \Dt)\}$ satisfies the semi-regularity condition \eqref{cond:semi-regular}, we obtain \eqref{eq:lem:inequarity_disc_norm}. 
\end{proof}

\begin{lemma}
	\label{lem:ineq:|Lap|}
	Assume that a family $\{(\{\PtSetT{\k}\}_{\k=1}^{\TStepMax}, \PvSet, \h, \Dt)\}$ satisfies the time step condition with $\ConstTStepCond$. 
	Then, for $\phi\in\FsSolAppT{\k}(\Lambda)^\Dim\,(\Lambda\subset\IndexSetAll,\, \k=0,1,\dots,\TStepMax)$, we have
	\begin{equation}
		\NormLISPH{\LapISPH{\phi}{}{\k}}{2}{\Lambda}^2\leq\dfrac{2\ConstTStepCond}{\Dt\KVisc}\NormHzTISPH{\phi}{\k}{\Lambda}^2,\qquad\TStep=0,1,\dots,\TStepMax. 
		\label{ineq:|Lap|}
	\end{equation}
	Here, the definition of $\LapISPH{\phi}{}{\k}$ is analogous to that given for $\LapISPH{\psi}{}{\k}$ in \eqref{def:LapISPH}, with the vector function $\psi$ replaced by the scalar function $\phi$. 
\end{lemma}

\begin{proof}
	From the Cauchy--Schwarz inequality (see Appendix \ref{appendix:math_tool}), we have 
	\begin{align}
		\NormLISPH{\LapISPH{\phi}{}{\k}}{2}{\Lambda}^2
		&=\sum_{\i\in\Lambda}\Pv{\i}\BrS{2\sum_{\j\in\Lambda\setminus\{\i\}}\Pv{\j}\frac{\phi_{\i}-\phi_{\j}}{|\PtT{\i}{\k}-\PtT{\j}{\k}|}\frac{\PtT{\i}{\k}-\PtT{\j}{\k}}{|\PtT{\i}{\k}-\PtT{\j}{\k}|}\cdot\Grad\wh(|\PtT{\i}{\k}-\PtT{\j}{\k}|)}^2
		\nn\\
		&=4\sum_{\i\in\Lambda}\Pv{\i}\BrS{\sum_{\j\in\Lambda\setminus\{\i\}}\Pv{\j}\frac{\phi_{\i}-\phi_{\j}}{|\PtT{\i}{\k}-\PtT{\j}{\k}|}\BrANd{\whd(|\PtT{\i}{\k}-\PtT{\j}{\k}|)}}^2
		\nn\\
		&\leq 4\sum_{\i\in\Lambda}\Pv{\i}\BrS{\sum_{\j\in\Lambda\setminus\{\i\}}\Pv{\j}\frac{|\phi_{\i}-\phi_{\j}|^2}{|\PtT{\i}{\k}-\PtT{\j}{\k}|}\BrANd{\whd(|\PtT{\i}{\k}-\PtT{\j}{\k}|)}}\BrS{\sum_{\j\in\Lambda\setminus\{\i\}}\Pv{\j}\frac{\BrANd{\whd(|\PtT{\i}{\k}-\PtT{\j}{\k}|)}}{|\PtT{\i}{\k}-\PtT{\j}{\k}|}}
		\nn\\
		&\leq 4\max_{\i=1,2,\dots,\N}\BrS{\sum_{\j\in\Lambda\setminus\{\i\}}\Pv{\j}\frac{\BrANd{\whd(|\PtT{\i}{\k}-\PtT{\j}{\k}|)}}{|\PtT{\i}{\k}-\PtT{\j}{\k}|}}\NormHzTISPH{\phi}{\k}{\Lambda}^2. 
	\end{align}
	As the family $\{(\{\PtSetT{\k}\}_{\k=1}^{\TStepMax}, \PvSet, \h, \Dt)\}$ satisfies the time step condition \eqref{cond:time-step}, we obtain \eqref{ineq:|Lap|}. 
\end{proof}

\subsection{Stability for the implicit scheme}
\label{subsec:stab_implicit_scheme}
From the lemmas in the previous section, we obtain the following stability of velocity in two-dimensional space for the implicit scheme in the ISPH method.  
\begin{theorem}
	\label{thm:stab_im}
	Let $\Dim=2$. 
	Let $(\VelApp{}{\k+1},\PresApp{}{\k+1})$ be the solution of the implicit scheme in the ISPH method. 
	Assume a family $\{(\{\PtSetT{\k}\}_{\k=1}^{\TStepMax}, \PvSet, \h, \Dt)\}$ that satisfies the semi-regularity condition with $\ConstSemiReg$ and whose particle distribution $\PtSetT{\k}$ satisfies the $\h$-connectivity condition. 
	Then, there exists a positive constant $\gconst$ dependent only on $\TimeMax$, $\KVisc$, and $\ConstSemiReg$ such that
	\begin{equation}
		\NormLISPH{\VelApp{}{\k+1}}{2}{\IndexSetAll}^2
		\leq\gconst\BrS{\NormLISPH{\VelIni}{2}{\IndexSetAll}^2+\sum_{\l=0}^\k\Dt\NormHzInvTISPH{\DmFunc_{}^{\l}}{\l}{\IndexSetAll}^2},\qquad \k=0,1,\dots,\TStepMax-1. 
		\label{thm:stab_im:ineq}
	\end{equation}
\end{theorem}

\begin{proof}
	From \eqref{implicit:correction} and Lemma \ref{lem:inequarity_disc_norm}, we have
	\begin{align}
		\NormLISPH{\VelApp{}{\k+1}}{2}{\IndexSetAll}^2
		&=\NormLISPH{\VelApp{}{\k+1}}{2}{\IndexSetFlT{\k}\cup\IndexSetSfT{\k}}^2=\InnerProdISPH{\VelApp{}{\k+1}}{\VelApp{}{\k+1}}{\IndexSetFlT{\k}\cup\IndexSetSfT{\k}}
		\nn\\
		&=\InnerProdISPH{\VelPred{}{\k+1}-\dfrac{\Dt}{\Dens}\PresAppGrad{}{\k+1}}{\VelPred{}{\k+1}-\dfrac{\Dt}{\Dens}\PresAppGrad{}{\k+1}}{\IndexSetFlT{\k}\cup\IndexSetSfT{\k}}
		\nn\\
		&=\NormLISPH{\VelPred{}{\k+1}}{2}{\IndexSetFlT{\k}\cup\IndexSetSfT{\k}}^2+\dfrac{\Dt^2}{\Dens^2}\NormLISPH{\PresAppGrad{}{\k+1}}{2}{\IndexSetFlT{\k}\cup\IndexSetSfT{\k}}^2-2\dfrac{\Dt}{\Dens}\InnerProdISPH{\PresAppGrad{}{\k+1}}{\VelPred{}{\k+1}}{\IndexSetFlT{\k}\cup\IndexSetSfT{\k}}.
		\label{prf:thm:stab_im:VelApp_k+1:01}
	\end{align}
	From \eqref{implicit:Poisson} and Lemmas \ref{lem:propaties_disc_norm:03}--\ref{lem:inequarity_disc_norm}, we have
	\begin{align}
		\NormLISPH{\PresAppGrad{}{\k+1}}{2}{\IndexSetFlT{\k}\cup\IndexSetSfT{\k}}^2
		&\leq(2+\ConstSemiReg\Dt)\NormHzTISPH{\PresApp{}{\k+1}}{\k}{\IndexSetFlT{\k}\cup\IndexSetSfT{\k}}^2
		\nn\\
		&=-(2+\ConstSemiReg\Dt)\InnerProdISPH{\PresApp{}{\k+1}}{\PresAppLap{}{\k+1}}{\IndexSetFlT{\k}\cup\IndexSetSfT{\k}}
		\nn\\
		&=-(2+\ConstSemiReg\Dt)\dfrac{\Dens}{\Dt}\InnerProdISPH{\PresApp{}{\k+1}}{\VelPredDiv{}{\k+1}}{\IndexSetFlT{\k}\cup\IndexSetSfT{\k}}
		\nn\\
		&=(2+\ConstSemiReg\Dt)\dfrac{\Dens}{\Dt}\InnerProdISPH{\PresAppGrad{}{\k+1}}{\VelPred{}{\k+1}}{\IndexSetFlT{\k}\cup\IndexSetSfT{\k}}. \label{prf:thm:stab_im:PresApp_k+1:01}
	\end{align}
	From \eqref{prf:thm:stab_im:VelApp_k+1:01}--\eqref{prf:thm:stab_im:PresApp_k+1:01} and Lemma \ref{lem:propaties_disc_norm:01}, we have 
	\begin{align}
		\NormLISPH{\VelApp{}{\k+1}}{2}{\IndexSetAll}^2
		&\leq\NormLISPH{\VelPred{}{\k+1}}{2}{\IndexSetFlT{\k}\cup\IndexSetSfT{\k}}^2+\ConstSemiReg\Dt\dfrac{\Dt}{\Dens}\InnerProdISPH{\PresAppGrad{}{\k+1}}{\VelPred{}{\k+1}}{\IndexSetFlT{\k}\cup\IndexSetSfT{\k}}
		\nn\\
		&\leq\NormLISPH{\VelPred{}{\k+1}}{2}{\IndexSetFlT{\k}\cup\IndexSetSfT{\k}}^2+\ConstSemiReg\Dt\dfrac{\Dt}{\Dens}\NormLISPH{\PresAppGrad{}{\k+1}}{2}{\IndexSetFlT{\k}\cup\IndexSetSfT{\k}}\NormLISPH{\VelPred{}{\k+1}}{2}{\IndexSetFlT{\k}\cup\IndexSetSfT{\k}}
		\nn\\
		&=\NormLISPH{\VelPred{}{\k+1}}{2}{\IndexSetAll}^2+\ConstSemiReg\Dt\dfrac{\Dt}{\Dens}\NormLISPH{\PresAppGrad{}{\k+1}}{2}{\IndexSetFlT{\k}\cup\IndexSetSfT{\k}}\NormLISPH{\VelPred{}{\k+1}}{2}{\IndexSetAll}.
		\label{prf:thm:stab_im:VelApp_k+1:02}
	\end{align}
	Moreover, from \eqref{prf:thm:stab_im:PresApp_k+1:01} and Lemma \ref{lem:propaties_disc_norm:01}, we have
	\begin{align}
		\NormLISPH{\PresAppGrad{}{\k+1}}{2}{\IndexSetFlT{\k}\cup\IndexSetSfT{\k}}^2
		&\leq(2+\ConstSemiReg\Dt)\dfrac{\Dens}{\Dt}\NormLISPH{\PresAppGrad{}{\k+1}}{2}{\IndexSetFlT{\k}\cup\IndexSetSfT{\k}}\NormLISPH{\VelPred{}{\k+1}}{2}{\IndexSetFlT{\k}\cup\IndexSetSfT{\k}}. 
	\end{align}
	Therefore, we have
	\begin{equation}
		\NormLISPH{\PresAppGrad{}{\k+1}}{2}{\IndexSetFlT{\k}\cup\IndexSetSfT{\k}}\leq(2+\ConstSemiReg\Dt)\dfrac{\Dens}{\Dt}\NormLISPH{\VelPred{}{\k+1}}{2}{\IndexSetFlT{\k}\cup\IndexSetSfT{\k}}=(2+\ConstSemiReg\Dt)\dfrac{\Dens}{\Dt}\NormLISPH{\VelPred{}{\k+1}}{2}{\IndexSetAll}. 
		\label{prf:thm:stab_im:PresApp_k+1:02}
	\end{equation}
	By combining \eqref{prf:thm:stab_im:VelApp_k+1:02} and \eqref{prf:thm:stab_im:PresApp_k+1:02}, we obtain
	\begin{align}
		\NormLISPH{\VelApp{}{\k+1}}{2}{\IndexSetAll}^2
		&\leq\{1+\ConstSemiReg(2+\ConstSemiReg\Dt)\Dt\}\NormLISPH{\VelPred{}{\k+1}}{2}{\IndexSetAll}^2. 
		\label{prf:thm:stab_im:VelApp_k+1:03}
	\end{align}
	From \eqref{implicit:prediction} and Lemmas \ref{lem:propaties_disc_norm:01}--\ref{lem:propaties_disc_norm:03}, we have
	\begin{align}
		\NormLISPH{\VelPred{}{\k+1}}{2}{\IndexSetAll}^2
		&=\NormLISPH{\VelPred{}{\k+1}}{2}{\IndexSetFlT{\k}\cup\IndexSetSfT{\k}}^2=\InnerProdISPH{\VelPred{}{\k+1}}{\VelPred{}{\k+1}}{\IndexSetFlT{\k}\cup\IndexSetSfT{\k}}
		\nn\\
		&=\InnerProdISPH{\VelApp{}{\k}+\Dt\KVisc\VelPredLap{}{\k+1}+\Dt\DmFunc_{}^{\k}}{\VelPred{}{\k+1}}{\IndexSetFlT{\k}\cup\IndexSetSfT{\k}}
		\nn\\
		&=\InnerProdISPH{\VelApp{}{\k}}{\VelPred{}{\k+1}}{\IndexSetFlT{\k}\cup\IndexSetSfT{\k}}+\Dt\InnerProdISPH{\DmFunc_{}^{\k}}{\VelPred{}{\k+1}}{\IndexSetFlT{\k}\cup\IndexSetSfT{\k}}+\Dt\KVisc\InnerProdISPH{\VelPredLap{}{\k+1}}{\VelPred{}{\k+1}}{\IndexSetFlT{\k}\cup\IndexSetSfT{\k}}
		\nn\\
		&\leq\NormLISPH{\VelApp{}{\k}}{2}{\IndexSetFlT{\k}\cup\IndexSetSfT{\k}}\NormLISPH{\VelPred{}{\k+1}}{2}{\IndexSetFlT{\k}\cup\IndexSetSfT{\k}}+\Dt\NormHzInvTISPH{\DmFunc_{}^{\k}}{\k}{\IndexSetFlT{\k}\cup\IndexSetSfT{\k}}\NormHzTISPH{\VelPred{}{\k+1}}{\k}{\IndexSetFlT{\k}\cup\IndexSetSfT{\k}}-\Dt\KVisc\NormHzTISPH{\VelPred{}{\k+1}}{\k}{\IndexSetFlT{\k}\cup\IndexSetSfT{\k}}^2.
	\end{align}
	For $\alpha, \beta\in\dR$ and $s>0$, the following inequality holds: 
	\begin{align}
		\alpha\beta\leq \dfrac{s}{2}\alpha^2+\dfrac{1}{2s}\beta^2. 
		\label{ineq:ab}
	\end{align}
	Hence, by utilizing
	\begin{align}
		\NormLISPH{\VelApp{}{\k}}{2}{\IndexSetFlT{\k}\cup\IndexSetSfT{\k}}\NormLISPH{\VelPred{}{\k+1}}{2}{\IndexSetFlT{\k}\cup\IndexSetSfT{\k}}
		&\leq\dfrac{1}{2}\NormLISPH{\VelApp{}{\k}}{2}{\IndexSetFlT{\k}\cup\IndexSetSfT{\k}}^2+\dfrac{1}{2}\NormLISPH{\VelPred{}{\k+1}}{2}{\IndexSetFlT{\k}\cup\IndexSetSfT{\k}}^2=\dfrac{1}{2}\NormLISPH{\VelApp{}{\k}}{2}{\IndexSetAll}^2+\dfrac{1}{2}\NormLISPH{\VelPred{}{\k+1}}{2}{\IndexSetAll}^2,
		\\
		\NormHzInvTISPH{\DmFunc_{}^{\k}}{\k}{\IndexSetFlT{\k}\cup\IndexSetSfT{\k}}\NormHzTISPH{\VelPred{}{\k+1}}{\k}{\IndexSetFlT{\k}\cup\IndexSetSfT{\k}}
		&\leq\dfrac{1}{4\KVisc}\NormHzInvTISPH{\DmFunc_{}^{\k}}{\k}{\IndexSetFlT{\k}\cup\IndexSetSfT{\k}}^2+\KVisc\NormHzTISPH{\VelPred{}{\k+1}}{\k}{\IndexSetFlT{\k}\cup\IndexSetSfT{\k}}^2, 
	\end{align}
	we have
	\begin{align}
		\NormLISPH{\VelPred{}{\k+1}}{2}{\IndexSetAll}^2
		&\leq\dfrac{1}{2}\NormLISPH{\VelApp{}{\k}}{2}{\IndexSetAll}^2+\dfrac{1}{2}\NormLISPH{\VelPred{}{\k+1}}{2}{\IndexSetAll}^2+\dfrac{\Dt}{4\KVisc}\NormHzInvTISPH{\DmFunc_{}^{\k}}{\k}{\IndexSetFlT{\k}\cup\IndexSetSfT{\k}}^2. 
	\end{align}
	Thus, we have
	\begin{equation}
		\NormLISPH{\VelPred{}{\k+1}}{2}{\IndexSetAll}^2\leq\NormLISPH{\VelApp{}{\k}}{2}{\IndexSetAll}^2+\dfrac{\Dt}{2\KVisc}\NormHzInvTISPH{\DmFunc_{}^{\k}}{\k}{\IndexSetFlT{\k}\cup\IndexSetSfT{\k}}^2\leq\NormLISPH{\VelApp{}{\k}}{2}{\IndexSetAll}^2+\dfrac{\Dt}{2\KVisc}\NormHzInvTISPH{\DmFunc_{}^{\k}}{\k}{\IndexSetAll}^2.
		\label{prf:thm:stab_im:VelPred_k+1:01}
	\end{equation}
	From \eqref{prf:thm:stab_im:VelApp_k+1:03} and \eqref{prf:thm:stab_im:VelPred_k+1:01}, we have
	\begin{align}
		\NormLISPH{\VelApp{}{\k+1}}{2}{\IndexSetAll}^2
		&\leq\NormLISPH{\VelApp{}{\k}}{2}{\IndexSetAll}^2+\ConstSemiReg(2+\ConstSemiReg\Dt)\Dt\NormLISPH{\VelApp{}{\k}}{2}{\IndexSetAll}+(1+\ConstSemiReg(2+\ConstSemiReg\Dt)\Dt)\dfrac{\Dt}{2\KVisc}\NormHzInvTISPH{\DmFunc_{}^{\k}}{\k}{\IndexSetAll}^2. 
		\label{prf:thm:stab_im:Vel_k+1:01}
	\end{align}
	By replacing the index $\k$ with $\l$ in \eqref{prf:thm:stab_im:Vel_k+1:01} and summing it over $\l=0$ to $\k$, we have
	\begin{equation}
		\NormLISPH{\VelApp{}{\k+1}}{2}{\IndexSetAll}^2
		\leq\NormLISPH{\VelIni}{2}{\IndexSetAll}^2+\ConstSemiReg(2+\ConstSemiReg\Dt)\sum_{\l=0}^\k\Dt\NormLISPH{\VelApp{}{\l}}{2}{\IndexSetAll}+\dfrac{1+\ConstSemiReg(2+\ConstSemiReg\Dt)\Dt}{2\KVisc}\sum_{\l=0}^\k\Dt\NormHzInvTISPH{\DmFunc_{}^{\l}}{\l}{\IndexSetAll}^2.
	\end{equation}
	Because $\Dt<\TimeMax$, we have 
	\begin{equation}
		\NormLISPH{\VelApp{}{\k+1}}{2}{\IndexSetAll}^2
		\leq\NormLISPH{\VelIni}{2}{\IndexSetAll}^2+\ConstSemiReg(2+\ConstSemiReg\TimeMax)\sum_{\l=0}^\k\Dt\NormLISPH{\VelApp{}{\l}}{2}{\IndexSetAll}+\dfrac{1+\ConstSemiReg(2+\ConstSemiReg\TimeMax)\TimeMax}{2\KVisc}\sum_{\l=0}^\k\Dt\NormHzInvTISPH{\DmFunc_{}^{\l}}{\l}{\IndexSetAll}^2. 
	\end{equation}
	Consequently, Gr\"onwall's inequality (see Appendix \ref{appendix:math_tool}) yields
	\begin{equation}
		\NormLISPH{\VelApp{}{\k+1}}{2}{\IndexSetAll}^2
		\leq\exp\BrS{\ConstSemiReg(2+\ConstSemiReg\TimeMax)\TimeMax}\left(\NormLISPH{\VelIni}{2}{\IndexSetAll}^2+\dfrac{1+\ConstSemiReg(2+\ConstSemiReg\TimeMax)\TimeMax}{2\KVisc}\sum_{l=1}^\k\Dt\NormHzInvTISPH{\DmFunc_{}^{\l}}{\l}{\IndexSetAll}^2\right). 
	\end{equation}
	Taking $c$ as
	\begin{equation}
		c=\exp\BrS{\ConstSemiReg(2+\ConstSemiReg\TimeMax)\TimeMax}\max\BrM{1,\dfrac{1+\ConstSemiReg(2+\ConstSemiReg\TimeMax)\TimeMax}{2\KVisc}}, 
	\end{equation}
	we conclude \eqref{thm:stab_im:ineq}. 
\end{proof}

\subsection{Stability for the semi-implicit scheme}
\label{subsec:stab_semi-implicit_scheme}
In addition to the stability for the implicit scheme, we obtain the following stability of velocity in two-dimensional space for the semi-implicit scheme in the ISPH method.  
\begin{theorem}
	\label{thm:stab_semi}
	Let $\Dim=2$. 
	Let $(\VelApp{}{\k+1},\PresApp{}{\k+1})$ be the solution of the semi-implicit scheme in the ISPH method. 
	Assume a family $\{(\{\PtSetT{\k}\}_{\k=1}^{\TStepMax}, \PvSet, \h, \Dt)\}$ that satisfies the semi-regularity condition with $\ConstSemiReg$ and time step condition with $\ConstTStepCond$ and whose particle distribution $\PtSetT{\k}\,(\k=1,2,\dots,\TStepMax)$ satisfies the $\h$-connectivity condition. 
	Then, there exists a positive constant $\gconst$ dependent only on $\TimeMax$, $\KVisc$, $\ConstSemiReg$, and $\ConstTStepCond$ such that
	\begin{equation}
		\NormLISPH{\VelApp{}{\k+1}}{2}{\IndexSetAll}^2
		\leq\gconst\BrS{\NormLISPH{\VelIni}{2}{\IndexSetAll}^2+\Dt\sum_{\l=0}^\k\Dt\NormLISPH{\DmFunc_{}^{\l}}{2}{\IndexSetAll}^2+\sum_{\l=0}^\k\Dt\NormHzInvTISPH{\DmFunc_{}^{\l}}{\l}{\IndexSetAll}^2},\qquad \k=0,1,\dots,\TStepMax-1. 
		\label{thm:stab_semi:ineq}
	\end{equation}
\end{theorem}

\begin{proof}
	By the same procedure as for the estimation of $\VelApp{}{\k+1}$ in the proof of Theorem \ref{thm:stab_im}, we obtain
	\begin{align}
		\NormLISPH{\VelApp{}{\k+1}}{2}{\IndexSetAll}^2
		&\leq\{1+\ConstSemiReg(2+\ConstSemiReg\Dt)\Dt\}\NormLISPH{\VelPred{}{\k+1}}{2}{\IndexSetAll}^2. 
		\label{prf:thm:stab_semi:VelApp_k+1:02}
	\end{align}
	From \eqref{implicit:prediction} and Lemmas \ref{lem:propaties_disc_norm:01}--\ref{lem:propaties_disc_norm:02}, we have
	\begin{align}
		\NormLISPH{\VelPred{}{\k+1}}{2}{\IndexSetAll}^2
		&=\NormLISPH{\VelPred{}{\k+1}}{2}{\IndexSetFlT{\k}\cup\IndexSetSfT{\k}}^2=\InnerProdISPH{\VelPred{}{\k+1}}{\VelPred{}{\k+1}}{\IndexSetFlT{\k}\cup\IndexSetSfT{\k}}
		\nn\\
		&=\InnerProdISPH{\VelApp{}{\k}+\Dt\KVisc\VelAppLap{}{\k}+\Dt\DmFunc_{}^{\k}}{\VelApp{}{\k}+\Dt\KVisc\VelAppLap{}{\k}+\Dt\DmFunc_{}^{\k}}{\IndexSetFlT{\k}\cup\IndexSetSfT{\k}}
		\nn\\
		&=\NormLISPH{\VelApp{}{\k}}{2}{\IndexSetFlT{\k}\cup\IndexSetSfT{\k}}^2+\Dt^2\KVisc^2\NormLISPH{\VelAppLap{}{\k}}{2}{\IndexSetFlT{\k}\cup\IndexSetSfT{\k}}^2+\Dt^2\NormLISPH{\DmFunc_{}^{\k}}{2}{\IndexSetFlT{\k}\cup\IndexSetSfT{\k}}^2
		\nn\\
		&\quad+2\Dt\KVisc\InnerProdISPH{\VelApp{}{\k}}{\VelAppLap{}{\k}}{\IndexSetFlT{\k}\cup\IndexSetSfT{\k}}+2\Dt^2\KVisc\InnerProdISPH{\VelAppLap{}{\k}}{\DmFunc_{}^{\k}}{\IndexSetFlT{\k}\cup\IndexSetSfT{\k}}
		+2\Dt\InnerProdISPH{\DmFunc_{}^{\k}}{\VelApp{}{\k}}{\IndexSetFlT{\k}\cup\IndexSetSfT{\k}}
		\nn\\
		&\leq\NormLISPH{\VelApp{}{\k}}{2}{\IndexSetAll}^2+\Dt^2\KVisc^2\NormLISPH{\VelAppLap{}{\k}}{2}{\IndexSetFlT{\k}\cup\IndexSetSfT{\k}}^2+\Dt^2\NormLISPH{\DmFunc_{}^{\k}}{2}{\IndexSetFlT{\k}\cup\IndexSetSfT{\k}}^2-2\Dt\KVisc\NormHzTISPH{\VelApp{}{\k}}{\k}{\IndexSetFlT{\k}\cup\IndexSetSfT{\k}}^2\nn\\
		&\quad+2\Dt^2\KVisc\NormLISPH{\VelAppLap{}{\k}}{2}{\IndexSetFlT{\k}\cup\IndexSetSfT{\k}}\NormLISPH{\DmFunc_{}^{\k}}{2}{\IndexSetFlT{\k}\cup\IndexSetSfT{\k}}+2\Dt\NormHzInvTISPH{\DmFunc_{}^{\k}}{\k}{\IndexSetFlT{\k}\cup\IndexSetSfT{\k}}\NormHzTISPH{\VelApp{}{\k}}{\k}{\IndexSetFlT{\k}\cup\IndexSetSfT{\k}}. 
	\end{align}
	From \eqref{ineq:ab}, for $\ConstTStepCond\,(0<\ConstTStepCond<1)$, we have
	\begin{align}
		\NormLISPH{\VelAppLap{}{\k}}{2}{\IndexSetFlT{\k}\cup\IndexSetSfT{\k}}\NormLISPH{\DmFunc_{}^{\k}}{2}{\IndexSetFlT{\k}\cup\IndexSetSfT{\k}}
		&\leq\dfrac{\KVisc(1-\ConstTStepCond)}{4\ConstTStepCond}\NormLISPH{\VelAppLap{}{\k}}{2}{\IndexSetFlT{\k}\cup\IndexSetSfT{\k}}^2+\dfrac{\ConstTStepCond}{\KVisc(1-\ConstTStepCond)}\NormLISPH{\DmFunc_{}^{\k}}{2}{\IndexSetFlT{\k}\cup\IndexSetSfT{\k}}^2,
		\\
		\NormHzInvTISPH{\DmFunc_{}^{\k}}{\k}{\IndexSetFlT{\k}\cup\IndexSetSfT{\k}}\NormHzTISPH{\VelApp{}{\k}}{\k}{\IndexSetFlT{\k}\cup\IndexSetSfT{\k}}
		&\leq\dfrac{1}{2\KVisc(1-\ConstTStepCond)}\NormHzInvTISPH{\DmFunc_{}^{\k}}{\k}{\IndexSetFlT{\k}\cup\IndexSetSfT{\k}}^2+\dfrac{\KVisc(1-\ConstTStepCond)}{2}\NormHzTISPH{\VelApp{}{\k}}{\k}{\IndexSetFlT{\k}\cup\IndexSetSfT{\k}}^2. 
	\end{align}
	Hence, we have
	\begin{align}
		\NormLISPH{\VelPred{}{\k+1}}{2}{\IndexSetAll}^2
		&\leq\NormLISPH{\VelApp{}{\k}}{2}{\IndexSetAll}^2+\Dt^2\dfrac{1+\ConstTStepCond}{1-\ConstTStepCond}\NormLISPH{\DmFunc_{}^{\k}}{2}{\IndexSetFlT{\k}\cup\IndexSetSfT{\k}}^2+\dfrac{\Dt}{\KVisc(1-\ConstTStepCond)}\NormHzInvTISPH{\DmFunc_{}^{\k}}{\k}{\IndexSetFlT{\k}\cup\IndexSetSfT{\k}}^2
		\nn\\
		&\quad+\Dt\KVisc(1+\ConstTStepCond)\BrS{\dfrac{\Dt\KVisc}{2\ConstTStepCond}\NormLISPH{\VelAppLap{}{\k}}{2}{\IndexSetFlT{\k}\cup\IndexSetSfT{\k}}^2-\NormHzTISPH{\VelApp{}{\k}}{\k}{\IndexSetFlT{\k}\cup\IndexSetSfT{\k}}^2}. 
	\end{align}
	By Lemma \ref{lem:ineq:|Lap|}, we have 
	\begin{align}
		\dfrac{\Dt\KVisc}{2\ConstTStepCond}\NormLISPH{\VelAppLap{}{\k}}{2}{\IndexSetFlT{\k}\cup\IndexSetSfT{\k}}^2-\NormHzTISPH{\VelApp{}{\k}}{\k}{\IndexSetFlT{\k}\cup\IndexSetSfT{\k}}^2\leq 0. 
	\end{align}
	Hence, we obtain
	\begin{align}
		\NormLISPH{\VelPred{}{\k+1}}{2}{\IndexSetAll}^2
		&\leq\NormLISPH{\VelApp{}{\k}}{2}{\IndexSetFlT{\k}\cup\IndexSetSfT{\k}}^2+\Dt^2\dfrac{1+\ConstTStepCond}{1-\ConstTStepCond}\NormLISPH{\DmFunc_{}^{\k}}{2}{\IndexSetFlT{\k}\cup\IndexSetSfT{\k}}^2+\dfrac{\Dt}{\KVisc(1-\ConstTStepCond)}\NormHzInvTISPH{\DmFunc_{}^{\k}}{\k}{\IndexSetFlT{\k}\cup\IndexSetSfT{\k}}^2
		\nn\\
		&\leq\NormLISPH{\VelApp{}{\k}}{2}{\IndexSetAll}^2+\Dt^2\dfrac{1+\ConstTStepCond}{1-\ConstTStepCond}\NormLISPH{\DmFunc_{}^{\k}}{2}{\IndexSetAll}^2+\dfrac{\Dt}{\KVisc(1-\ConstTStepCond)}\NormHzInvTISPH{\DmFunc_{}^{\k}}{\k}{\IndexSetAll}^2.
		\label{prf:thm:stab_semi:VelPred_k+1:01}
	\end{align}
	From \eqref{prf:thm:stab_semi:VelApp_k+1:02} and \eqref{prf:thm:stab_semi:VelPred_k+1:01}, we have
	\begin{multline}
		\NormLISPH{\VelApp{}{\k+1}}{2}{\IndexSetAll}^2
		\leq\NormLISPH{\VelApp{}{\k}}{2}{\IndexSetAll}^2+\ConstSemiReg(2+\ConstSemiReg\Dt)\Dt\NormLISPH{\VelApp{}{\k}}{2}{\IndexSetAll}
		\\
		+\{1+\ConstSemiReg(2+\ConstSemiReg\Dt)\Dt\}\BrM{\Dt^2\dfrac{1+\ConstTStepCond}{1-\ConstTStepCond}\NormLISPH{\DmFunc_{}^{\k}}{2}{\IndexSetAll}^2+\dfrac{\Dt}{\KVisc(1-\ConstTStepCond)}\NormHzInvTISPH{\DmFunc_{}^{\k}}{\k}{\IndexSetAll}^2}. 
		\label{prf:thm:stab_semi:Vel_k+1:01}
	\end{multline}
	By replacing the index $\k$ with $\l$ in \eqref{prf:thm:stab_semi:Vel_k+1:01} and summing it over $\l=0$ to $\k$, we have
	\begin{multline}
		\NormLISPH{\VelApp{}{\k+1}}{2}{\IndexSetAll}^2
		\leq\NormLISPH{\VelIni}{2}{\IndexSetAll}^2+\ConstSemiReg(2+\ConstSemiReg\Dt)\sum_{\l=0}^\k\Dt\NormLISPH{\VelApp{}{\l}}{2}{\IndexSetAll}
		\\
		+\dfrac{1+\ConstSemiReg(2+\ConstSemiReg\TimeMax)\TimeMax}{1-\ConstTStepCond}\BrM{\Dt(1+\ConstTStepCond)\sum_{\l=0}^\k\Dt\NormLISPH{\DmFunc_{}^{\l}}{2}{\IndexSetAll}^2+\dfrac{1}{\KVisc}\sum_{\l=0}^\k\Dt\NormHzInvTISPH{\DmFunc_{}^{\l}}{\l}{\IndexSetAll}^2}. 
	\end{multline}
	From $\Dt<\TimeMax$, we have 
	\begin{multline}
		\NormLISPH{\VelApp{}{\k+1}}{2}{\IndexSetAll}^2
		\leq\NormLISPH{\VelIni}{2}{\IndexSetAll}^2+\ConstSemiReg(2+\ConstSemiReg\TimeMax)\sum_{\l=0}^\k\Dt\NormLISPH{\VelApp{}{\l}}{2}{\IndexSetAll}\\
		+\dfrac{1+\ConstSemiReg(2+\ConstSemiReg\TimeMax)\TimeMax}{1-\ConstTStepCond}\BrM{\Dt(1+\ConstTStepCond)\sum_{\l=0}^\k\Dt\NormLISPH{\DmFunc_{}^{\l}}{2}{\IndexSetAll}^2+\dfrac{1}{\KVisc}\sum_{\l=0}^\k\Dt\NormHzInvTISPH{\DmFunc_{}^{\l}}{\l}{\IndexSetAll}^2}. 
	\end{multline}
	Consequently, Gr\"onwall's inequality (see Appendix \ref{appendix:math_tool}) yields
	\begin{multline}
		\NormLISPH{\VelApp{}{\k+1}}{2}{\IndexSetAll}^2
		\leq\exp\BrS{\ConstSemiReg(2+\ConstSemiReg\TimeMax)\TimeMax}\Bigg\{\NormLISPH{\VelIni}{2}{\IndexSetAll}^2
		\\
		\quad+\dfrac{1+\ConstSemiReg(2+\ConstSemiReg\TimeMax)\TimeMax}{1-\ConstTStepCond}\bigg(\Dt(1+\ConstTStepCond)\sum_{\l=0}^\k\Dt\NormLISPH{\DmFunc_{}^{\l}}{2}{\IndexSetAll}^2+\dfrac{1}{\KVisc}\sum_{\l=0}^\k\Dt\NormHzInvTISPH{\DmFunc_{}^{\l}}{\l}{\IndexSetAll}^2\bigg)\Bigg\}. 
		\label{prf:thm:stab_semi:VelApp_k+1:03}
	\end{multline}
	Taking $c$ as
	\begin{equation}
		c=\exp\BrS{\ConstSemiReg(2+\ConstSemiReg\TimeMax)\TimeMax}\max\BrM{1,(1+\ConstSemiReg(2+\ConstSemiReg\TimeMax)\TimeMax)\dfrac{1+\ConstTStepCond}{1-\ConstTStepCond},\dfrac{1+\ConstSemiReg(2+\ConstSemiReg\TimeMax)\TimeMax}{\KVisc(1-\ConstTStepCond)}}, 
	\end{equation}
	we conclude \eqref{thm:stab_semi:ineq}. 
\end{proof}

\subsection{Extension for modified schemes}
\label{sec:modified_schemes}
We consider improving the implicit and semi-implicit schemes by utilizing our results. 
As the time step $\Dt$ and particle volume set $\PvSet=\{\Pv{\i}\}$ are fixed in the previous sections, we consider the introduction of modified schemes with variable time step $\DtT{\k}$ and particle volume set $\PvSetT{\k}=\{\PvT{\i}{\k}\}$ defined so as to satisfy some key conditions. 

For $\k=0,1,\dots,\TStepMax-1$, let $\DtT{\k}>0$ be a variable time step satisfying
\begin{equation}
	\sum_{\k=0}^{\TStepMax-1}\DtT{\k}=\TimeMax. 
\end{equation}
Then, the $\TStep$th time $\tk$ is defined as $\tk=0\,(\k=0)$, and $\tkp\deq\tk+\DtT{\k}\,(\k=0,1,\dots,\TStepMax-1)$. 
We set the particle volume set $\PvSetT{\k}=\{\PvT{\i}{\k}\}$ by a solution of the linear equation
\begin{equation}
	A^\k \PvChar^\k=b^\k, 
	\label{def:mod_pv}
\end{equation}
where $A^\k\in\dR^{\N\times\N}$, $\PvChar^\k\in\dR^{\N}$, and $b^\k\in\dR^{\N}$ are
\begin{align}
	A^\k_{\i\j}&\deq\BrANd{\PtT{\i}{\k}-\PtT{\j}{\k}}\BrANd{\whd(|\PtT{\i}{\k}-\PtT{\j}{\k}|)},
	\\
	\PvChar^\k&\deq(\PvT{1}{\k}, \PvT{2}{\k},\dots,\PvT{\N}{\k})^{\rm T},
	\\
	b^\k&\deq(\Dim, \Dim,\dots,\Dim)^{\rm T}, 
\end{align}
respectively. 
Then, because the condition  
\begin{equation}
	\sum_{\j=1}^\N\PvT{\j}{\k}\BrANd{\PtT{\i}{\k}-\PtT{\j}{\k}}\BrANd{\whd(|\PtT{\i}{\k}-\PtT{\j}{\k}|)}=\Dim,\qquad \i=1,2,\dots,\N, 
\end{equation}
is satisfied, the semi-regularity condition \eqref{cond:semi-regular} is automatically satisfied at  $\tk$. 
Therefore, we obtain the following corollary: 
\begin{cor}
	\label{cor:stab_improved_implicit}
	Let $\Dim=2$. 
	Let $(\VelApp{}{\k+1},\PresApp{}{\k+1})$ be the solution of the modified implicit scheme, which is the implicit scheme whose time step $\Dt$ and particle volume set $\PvSet$ are replaced with variable time step $\DtT{\k}$ and particle volume set $\PvSetT{\k}$. 
	Assume for a family $\{(\{\PtSetT{\k}, \PvSetT{\k}\}_{\k=1}^{\TStepMax}, \h, \DtT{\k})\}$ that its particle distribution $\PtSetT{\k}$ satisfies the $\h$-connectivity condition and particle volume set $\PvSetT{\k}=\{\PvT{\i}{\k}>0\}$ exists for $\k=0,1,\dots,\TStepMax$.  
	Then, there exists a positive constant $\gconst$ dependent only on $\TimeMax$, $\KVisc$, and $\ConstSemiReg$ such that
	\begin{equation}
		\NormLISPH{\VelApp{}{\k+1}}{2}{\IndexSetAll}^2
		\leq\gconst\BrS{\NormLISPH{\VelIni}{2}{\IndexSetAll}^2+\sum_{\l=0}^\k\DtT{\l}\NormHzInvTISPH{\DmFunc_{}^{\l}}{\l}{\IndexSetAll}^2},\qquad \k=0,1,\dots,\TStepMax-1. 
		\label{cor:stab_im:ineq}
	\end{equation}
\end{cor}

Moreover, for fixed constant $\ConstTStepCond\,(0<\ConstTStepCond<1)$, we give $\Dt$ as
\begin{equation}
	\DtT{\k}=\dfrac{\ConstTStepCond}{2\KVisc}\BrL{\max_{\i=1,2,\dots,\N}\BrS{\sum_{\j\neq\i}\PvT{\j}{\k}\dfrac{\BrANd{\whd(|\PtT{\i}{\k}-\PtT{\j}{\k}|)}}{|\PtT{\i}{\k}-\PtT{\j}{\k}|}}}^{-1},\qquad \k=1,2,\dots,\TStepMax. 
	\label{setting:DtT}
\end{equation}
Then, the time step condition \eqref{cond:time-step} is automatically satisfied at each time step. 
Therefore, we obtain the following corollary: 

\begin{cor}
	\label{cor:stab_improved_semi}
	Let $\Dim=2$. 
	Let $(\VelApp{}{\k+1},\PresApp{}{\k+1})$ be the solution of the modified semi-implicit scheme, which is the semi-implicit scheme whose time step $\Dt$ and particle volume set $\PvSet$ are replaced with variable time step $\DtT{\k}$ and particle volume set $\PvSetT{\k}$. 
	Assume for a family $\{(\{\PtSetT{\k}, \PvSetT{\k}\}_{\k=1}^{\TStepMax}, \h, \DtT{\k})\}$ that its particle distribution $\PtSetT{\k}$ satisfies the $\h$-connectivity condition and particle volume set $\PvSetT{\k}=\{\PvT{\i}{\k}>0\}$ exists for $\k=0,1,\dots,\TStepMax$.  
	Then, there exists a positive constant $\gconst$ dependent only on $\TimeMax$, $\KVisc$, $\ConstSemiReg$, and $\ConstTStepCond$ such that
	\begin{equation}
		\NormLISPH{\VelApp{}{\k+1}}{2}{\IndexSetAll}^2
		\leq\gconst\BrS{\NormLISPH{\VelIni}{2}{\IndexSetAll}^2+\sum_{\l=0}^\k(\DtT{\l})^2\NormLISPH{\DmFunc_{}^{\l}}{2}{\IndexSetAll}^2+\sum_{\l=0}^\k\DtT{\l}\NormHzInvTISPH{\DmFunc_{}^{\l}}{\l}{\IndexSetAll}^2},\qquad
		\k=0,1,\dots,\TStepMax-1. 
		\label{cor:stab_semi:ineq}
	\end{equation}
\end{cor}

\section{Concluding remarks}
\label{sec:Concluding_remarks}
We have analyzed the unique solvability and stability of the implicit and semi-implicit schemes in the incompressible smoothed particle hydrodynamics (ISPH) method. 
Three key conditions were introduced for our analysis, the three conditions on discrete parameters, which are the $\h$-connectivity, semi-regularity, and time step conditions. 
With $\h$-connectivity, the unique solvability of the implicit and semi-implicit schemes was obtained in two- and three-dimensional space. 
With the $\h$-connectivity and semi-regularity conditions, the stability of velocity for the implicit scheme was established in two-dimensional space. 
Moreover, with the addition of the time step condition, the stability of velocity for the semi-implicit scheme was established in two-dimensional space. 
Thanks to these results, the conditions on discrete parameters sufficient for obtaining stable computing with the ISPH method are clarified. 

As an application of these results, we introduced modified implicit and semi-implicit schemes by redefining discrete parameters. 
By introducing the modified particle volume set, which imposes an additional constraint condition at each step, the modified implicit scheme becomes stable without the semi-regularity condition. 
Moreover, by introducing the variable time step, which is updated according to the particle distribution and particle volume set, the modified semi-implicit scheme becomes stable without the semi-regularity and time step conditions. 

As future work, we will extend the stability to that in three-dimensional space and with boundary conditions such as Neumann boundary conditions in the pressure Poisson equation. 
Moreover, we will investigate convergence for the ISPH method mathematically. 

\section*{Acknowledgments}
This work was supported by JSPS KAKENHI Grant Number 17K17585, JSPS A3 Foresight Program, and ``Joint Usage/Research Center for Interdisciplinary Large-scale Information Infrastructures'' in Japan (Project ID: jh180060-NAH).

\appendix
\section{Mathematical tools}
\label{appendix:math_tool}

\subsection*{Cauchy--Schwarz inequality}
Let $M\in\dN$. 
For all $a_i, b_i\in\dR~(i=1,2,\dots,M)$, the following, called the Cauchy--Schwarz inequality, holds:
\begin{equation}
	\sum_{i=1}^Ma_i b_i \leq \BrS{\sum_{i=1}^M a_i^2}^{1/2}\BrS{\sum_{i=1}^M b_i^2}^{1/2}. 
	\label{Cauchy_Schwarz_ineq_disc}
\end{equation}

\subsection*{Gr\"onwall's inequality}
Let $M\in\dN$. 
Assume that $a_i, b_i>0~(i=0,1,\dots,M)$, $c>0$ satisfy the inequality
\begin{equation}
	a_k \leq a_0 + c + \sum_{j=0}^{k-1} a_j b_j, \qquad k=1,2,\dots,M. 
\end{equation}
Then the following, called Gr\"onwall's inequality, holds: 
\begin{equation}
	a_k \leq (a_0+c)\prod_{j=0}^{k-1} (1+b_j) \leq (a_0+c) \exp\BrS{\sum_{j=0}^{k-1} b_j}, \qquad k=1,2,\dots,M. 
	\label{Gronwall_ineq_disc}
\end{equation}

\bibliographystyle{spmpsci}      

\end{document}